\documentclass[11pt,thmsa]{article}%
\usepackage{amsmath}
\usepackage{amsfonts}
\usepackage{amssymb}
\usepackage{graphicx}%
\newtheorem{theorem}{Theorem}[section]

\newtheorem{lemma}[theorem]{Lemma}

\newtheorem{proposition}[theorem]{Proposition}
\newtheorem{remark}[theorem]{Remark}

\newenvironment{proof}[1][Proof]{\noindent\textbf{#1.} }{\ \rule{0.5em}{0.5em}}

\begin{document}
\title{Normal homogeneous  Finsler spaces\thanks{Supported by NSFC (no. 11271216, 11271198, 11221091), State Scholarship Fund of CSC (no. 201408120020), Science and Technology Development Fund for Universities and Colleges in Tianjin (no. 20141005), Doctor fund of Tianjin Normal University (no. 52XB1305) and SRFDP of China}}
\author{Ming Xu$^1$ and Shaoqiang Deng$^2$ \thanks{Corresponding author. E-mail: dengsq@nankai.edu.cn}\\
\\
$^1$College of Mathematics\\
Tianjin Normal University\\
 Tianjin 300387, P.R. China\\
 \\
$^2$School of Mathematical Sciences and LPMC\\
Nankai University\\
Tianjin 300071, P.R. China}
\date{}
\maketitle

\begin{abstract}
In this paper, we study normal homogeneous Finsler spaces. We first define the notion of a normal homogeneous Finsler space, using the method of isometric submersion of Finsler metrics. Then we study the geometric properties. In particular,  we establish a   technique
to reduce the classification of normal homogeneous Finsler spaces of positive flag curvature to an algebraic problem.
The main result of this paper is a classification of positively curved normal homogeneous Finsler spaces. It turns out that a coset space $G/H$ admits a positively curved normal homogeneous Finsler metric if and only if it admits a positively curved normal homogeneous Riemannian metric. We will also give a complete description
of the coset spaces admitting  non-Riemannian positively curved  normal homogeneous Finsler spaces.



\medskip
\noindent\textbf{R\'{e}sum\'{e}}

Dans cet article, nous \'{e}tudions les espaces de Finsler homog\`{e}nes normales. Nous d\'{e}finissons d'abord la notion d'un espace homog\`{e}ne Finsler normale, en utilisant la m¨¦thode de l'immersion isom\'{e}trique de m\'{e}triques de Finsler. Ensuite, nous \'{e}tudions les propri\'{e}t\'{e}s g\'{e}om\'{e}triques.
En particulier, nous \'{e}tablissons une technique
pour r\'{e}duire la classification des espaces de Finsler homog¨¨nes normales de courbure du pavillon positif \`{a} un probl¨¨me alg\'{e}brique. Le r\'{e}sultat principal de cet article est une classification des espaces \`{a} courbure positive de Finsler homog¨¨nes normales. Il se trouve que d'un espace de coset $ G / H $ admet une m\'{e}trique homog\`{e}ne normale courbure positive Finsler si et seulement si il admet une courbure positive normale homog\`{e}ne m\'{e}trique riemannienne. Nous allons aussi donner une description compl\`{e}te
 des espaces de coset admettant non-riemanniennes espaces \`{a} courbure positive de Finsler homog\`{e}nes normales.
\end{abstract}


\section{Introduction}
The goal of this paper is to extend the  study of normal homogeneous Riemannian manifolds to  normal homogeneous Finsler spaces. In Riemannian geometry, normal homogeneous manifolds play very important roles in many topics. For example, the first examples of  Riemannian manifolds of positive sectional curvature, which  are not isometric to a rank one symmetric manifold,  are normal; see Berger's paper \cite{Ber61}. This result motivated the study of homogeneous Riemannian manifolds with positive curvature and eventually led to a complete classification; see \cite{AW75, BB76, Wallach1972}. Meanwhile, normal homogeneous Riemannian manifolds provide many new examples with fine properties. For example, Wang and Ziller
studied normal homogeneous Einstein manifolds and found many new examples of Einstein metrics; See
\cite{WZ}. The notion of normal homogeneous Riemannian manifolds has been generalized to several
classes of special homogeneous Riemannian manifolds, such as $\delta$-homogeneous metrics; see for example \cite{BN}.

We now recall the notion of a normal homogeneous Riemannian manifold. Let  $G$ be a (connected) compact Lie group and $H$
 a closed subgroup of $G$.  Suppose   $\langle\cdot,\cdot\rangle$ is a bi-invariant inner product on the Lie algebra $\mathfrak{g}$ of $G$ and
 \begin{equation}\label{reduct}
 \mathfrak{g}=\mathfrak{h}+\mathfrak{m}
 \end{equation}
 is the orthogonal decomposition, where $\mathfrak{h}$ is the Lie algebra of $H$ and $\mathfrak{m}$ is the orthogonal complement subspace of $\mathfrak{h}$ in $\mathfrak{g}$ with respect to $\langle\cdot,\cdot\rangle$.
 Then $\mathfrak{m}$ is invariant under the adjoint action of $H$. The restriction of $\langle\cdot,\cdot\rangle$ to $\mathfrak{m}$, which is an inner product on $\mathfrak{m}$,  is also $\mathrm{Ad}(H)$-invariant. This restricted inner product induces a $G$-invariant
 Riemannian metric on $G/H$ (see \cite{KN63}). A homogeneous Riemannian metric of this type is called normal.

This definition can not be directly generalized to the Finslerian case. In fact,   even if there exists a bi-invariant Finsler metric on the Lie group $G$, it does not give a decomposition of the Lie algebra as in (\ref{reduct}), since in a Finsler space we have no notion of orthogonality.
To get a natural generalization of the normality in Finsler geometry, we need the notion of isometric
submersion for Finsler metrics. Note that the bi-invariant inner product on $\mathfrak{g}$ defines a bi-invariant Riemannian metric on $G$. The normal Riemannian metric defined above can also be uniquely determined under the requirement that the projection $\pi: G\to G/H$ is a submersion with respect to this bi-invariant
Riemannian metric. On the other hand, in Finsler geometry the theory of isometric submersions has been established in \cite{PD01}, and this method can be used to define the normality in Finsler geometry.

After clarifying the definition of the normality in Finsler geometry, we study some fundamental
geometric properties of normal Finsler spaces. Then we consider normal homogeneous Finsler spaces with positive flag curvature. This leads  to a classification theorem:

\begin{theorem}\label{main1}
Let $G$ be a connected Lie group and $H$ a closed subgroup of $G$.
Suppose the coset space $G/H$ admits a normal homogeneous Finsler space
of positive flag curvature. Then up to equivalence (defined in Section 5)
$G/H$ must be one in the following list.
\begin{description}
\item{\rm (1)}\quad Riemannian symmetric coset spaces of rank 1, i.e., $S^{n-1}=\mathrm{SO}(n)/\mathrm{SO}(n-1)$, $\mathbb{C}\mathrm{P}^{n-1}=\mathrm{SU}(n)/S(\mathrm{U}(n-1)\times \mathrm{U}(1))$, $\mathbb{H}\mathrm{P}^{n-1}=\mathrm{Sp}(n)/\mathrm{Sp}(n-1)\mathrm{Sp}(1)$ and $\mathbb{O}\mathrm{P}^2=\mathrm{F}_4/\mathrm{Spin}(9)$.
\item{\rm (2)}\quad Other normal homogeneous Finsler spheres, i.e., $\mathrm{SU}(n)/\mathrm{SU}(n-1)$,
$\mathrm{U}(n)/\mathrm{U}(n-1)$, $\mathrm{Sp}(n)/\mathrm{Sp}(n-1)$, $\mathrm{Sp}(n)S^1/\mathrm{Sp}(n-1)S^1$, $\mathrm{Sp}(n)\mathrm{Sp}(1)/\mathrm{Sp}(n-1)\mathrm{Sp}(1)$,
$\mathrm{G}_2/\mathrm{SU}(3)$, $\mathrm{Spin}(7)/\mathrm{G}_2$, and $\mathrm{Spin}(9)/\mathrm{Spin}(7)$.
\item{\rm (3)}\quad  Exceptional ones, i.e., $\mathrm{SU}(3)\times \mathrm{SO}(3)/\mathrm{U}^*(2)$, $\mathrm{Sp}(2)/\mathrm{SU}(2)$ and
$\mathrm{SU}(5)/\mathrm{Sp}(2)S^1$.
\end{description}
\end{theorem}

Any of the above coset spaces  admit  normal homogeneous Riemannian metrics with positive
sectional curvature. The coset space $\mathrm{SU}(3)\times \mathrm{SO}(3)/\mathrm{U}^*(2)$ was found to admit invariant  normal Riemannian metrics by Wilking in \cite{WILKING1990}, and all the other spaces  were found by Berger in \cite{Ber61}.

This classification, combined with Berger's classification of normal homogeneous Riemannian manifolds of positive sectional curvature \cite{Ber61}, gives the following theorem.
\begin{theorem}\label{main2}
Let $G$ be a connected compact Lie group and $H$ a closed subgroup of $G$. Then there exists a $G$-invariant normal homogeneous
Finsler metric on $G/H$ with positive flag curvature if and only if there exists a normal homogeneous  Riemannian metric
on $G/H$ with positive sectional curvature.
\end{theorem}

It is well known that if $G/H$ is a rank one Riemannian symmetric coset space, then any
$G$-invariant Finsler metric must be Riemannian (see \cite{DE12}). Hence a rank one Riemannian
coset space does not admit any non-Riemannian normal homogeneous Finsler metric metric with positive curvature. It is therefore interesting to find out which coset space in the list
of Theorem \ref{main1} admits non-Riemannian normal homogeneous Finsler metrics with positive curvature. This is completely settled by the following theorem.

\begin{theorem}\label{main3}
Among the coset spaces in the list in Theorem \ref{main1}, any invariant normal
Finsler metric with positive flag curvature on the symmetric spaces of
rank $1$, the homogeneous spheres $S^3=\mathrm{SU}(2)/\mathrm{SU}(1)=\mathrm{Sp}(1)/\mathrm{Sp}(0)$,
$S^6=\mathrm{G}_2/\mathrm{SU}(3)$, and $S^7=\mathrm{Spin}(7)/\mathrm{G}_2$
 must be Riemannian. On the other hand, any of  the other spaces admits non-Riemannian
 normal homogeneous Finsler metrics with positive flag curvature.
\end{theorem}

In Section 2, we present some preliminaries on Finsler geometry.
In Section 3, we define the notion of normal homogeneous Finsler space, using the method of isometric submersion in Finsler geometry. Moreover, we study the fundamental geometric
properties of normal homogeneous spaces and  prove that any normal homogeneous
Finsler space has vanishing S-curvature and non-negative flag curvature.
 Sections 4, 5 and 6 are devoted to   classifying all
normal homogeneous spaces of positively flag curvature. This also gives a proof of Theorems \ref{main1} and \ref{main2}. Finally, in Section 7, we complete the proof of Theorem \ref{main3}.

\section{Preliminaries}
\subsection{Minkowski norms and Finsler metrics}
Let $\mathbf{V}$ be a  a real vector space of dimension $n$.  A {\it Minkowski norm} on $\mathbf{V}$ is a
continuous function $F:\mathbf{V}\rightarrow[0,+\infty)$ satisfying the following
conditions:
\begin{description}
\item{\rm (1)} $F$ is positive and smooth on $\mathbf{V}\backslash\{0\}$;
\item{\rm (2)} $F(\lambda y)=\lambda F(y)$ for any $\lambda>0$;
\item{\rm (3)} With respect to any linear coordinates $y=y^i e_i$, the Hessian matrix
\begin{equation}
(g_{ij}(y))=\left(\frac{1}{2}[F^2]_{y^i y^j}(y)\right)
\end{equation}
is positive definite at any $y\ne 0$.
\end{description}
The Hessian matrix $(g_{ij}(y))$ and its inverse $(g^{ij}(y))$ of a Minkowski norm
can be used to raise  or lower down indices of tensors on the vector space.

Given any $y\neq 0$, the Hessian matrix $(g_{ij}(y))$ defines an inner product
$\langle \cdot,\cdot\rangle_y$ (or $\langle \cdot,\cdot\rangle_y^F$ when $F$ needs to be
specified) on $\mathbf{V}$ by
$\langle u,v \rangle_y=g_{ij}(y)u^i v^j$, where $u=u^i e_i$ and $v=v^i e_i$. This inner product can also be
written as
\begin{equation}
\langle u,v\rangle_{y}=\frac{1}{2}\frac{\partial^2}{\partial s
\partial t}[F^2(y+su+tv)]|_{s=t=0},
\end{equation}
which is independent of the choice of the linear coordinates.

A Finsler metric $F$  on a smooth manifold $M$ is a continuous function
$F:TM\rightarrow [0,+\infty)$ such that it is positive and smooth on the slit
tangent bundle $TM\backslash 0$, and its restriction to each tangent space is
a Minkowski norm. We generally say that  $(M,F)$ is a {\it Finsler manifold} or
a {\it Finsler space}.

Important examples of Finsler metrics include
Riemannian metrics, Randers metrics, $(\alpha,\beta)$-metrics, etc.
Riemannian metrics are a special class of Finsler metrics whose
 Hessian matrices at each point only depends on $x\in M$ rather than $y\in T_x M$.
A Riemannian  metric can also be defined as a  global smooth section
$g_{ij}dx^i dx^j$ of $\mathrm{Sym}^2(T^* M)$.
Randers metrics are the  simplest  class of non-Riemannian metrics in Finsler geometry. They
are defined by $F=\alpha+\beta$, where $\alpha$ is a Riemannian metric and $\beta$ is
a 1-form. Randers metrics can be naturally generalized as $(\alpha,\beta)$-metrics which have
the form $F=\alpha\phi(\beta/\alpha)$, where $\phi$ is   a positive smooth function $\phi$ (see \cite{CS04}). In recent years, there have been a lot of  work on Randers metrics
and $(\alpha,\beta)$-metrics.


\subsection{Geodesic spray, geodesics and S-curvature}
Now we  recall some relevant  notations and terminologies.
Let $(M, F)$ be  a Finsler space and $(x^1,x^2,\cdots, x^n)$ be a local coordinate system on an open subset $U$ in $M$.  In the following, we usually use the {\it standard} local coordinates $(x^i,y^j)$ for the open subset $TU$ in the tangent bundle $TM$, where $y=y^j\partial_{x^j}\in T_x M$.

The geodesic spray is a smooth vector field $G$ globally defined on the slit tangent bundle $TM\backslash\, 0$. On  a standard
local coordinate system, $G$ can be given by
\begin{equation}
G=y^i\partial_{x^i}-2G^i\partial_{y^i},
\end{equation}
where
\begin{equation}
G^i=\frac{1}{4} g^{il}([F^2]_{x^k y^l}y^k - [F^2]_{x^l}).
\end{equation}

A curve $c(t)$ on $M$ is called a geodesic if $(c(t),\dot{c}(t))$ is an integral curve
of $G$.  On a  standard local coordinate system, a geodesic $c(t)=(c^i(t))$ can be  characterized as a curve satisfying the partial differential equations
\begin{equation}
\ddot{c}^i(t)+2 G^i(c(t),\dot{c}(t))=0.
\end{equation}

It is well known that the geodesic spray is tangent to the indicatrix bundle in $TM$. So
$F(\dot{c}(t))$ is a constant function when $c(t)$ is a geodesic.  Therefore
we only need to consider geodesics of nonzero constant speed.

Z.Shen defines the following important non-Riemannian curvature in terms of the geodesic spray, which is now generally called S-curvature.

Let $(M, F)$ be  a Finsler space and   $(x^1, \cdots, x^n, y^1,\cdots, y^n)$ a standard local coordinate system. The Busemann-Hausdorff volume form can be
defined as $dV_{BH}=\sigma(x)dx^1\cdots dx^n$, in which
$$
\sigma(x)=\frac{\omega_n}{\mbox{Vol}\{(y^i)\in\mathbb{R}^n|F(x,y^i\partial_{x^i})<1\}},
$$
where $\mbox{Vol}$ denotes the volume of a subset with respect to the standard Euclidian metric on
$\mathbb{R}^n$, and $\omega_n=\mbox{Vol}(B_n(1))$. It is easily seen that the Busemann-Hausdorff
form is globally defined and does not depend on the specific coordinate system. On the other hand,
although the coefficient function $\sigma(x)$ is only locally defined and
depends on the choice of local coordinates $x=(x^i)$, the distortion function
\begin{equation*}
\tau(x,y)=\ln\frac{\sqrt{\det(g_{ij}(x,y))}}{\sigma(x)}
\end{equation*}
on $TM\backslash 0$ is independent of the local coordinates and globally defined.
The S-curvature $S(x,y)$ on $TM\backslash 0$ is  defined as the derivative of
$\tau(x,y)$ in the direction of the geodesic spray $G(x,y)$.

\subsection{Riemannian curvature, flag curvature and totally geodesic submanifolds}

On a Finsler manifold, we have a similar notion of curvature as in the Riemannian case, which
is called the {\it Riemannian curvature}. It can be defined either by the Jacobi field
or the structure equation for the curvature of the Chern connection.

On  a standard local coordinate system
the Riemannian curvature can be given as a family of linear maps $R_y$ (or $R_y^F$ when the
metric needs to be specified), where   $y\in T_xM\backslash 0$, such that
$R_y=R^i_k(y)\partial_{x^i}\otimes dx^k:T_x M\rightarrow T_x M$, here
\begin{equation}\label{1000}
R^i_k(y)=2\partial_{x^k}G^i-y^j\partial^2_{x^j y^k}G^i+2G^j\partial^2_{y^j y^k}G^i
-\partial_{y^j}G^i\partial_{y^k}G^j.
\end{equation}
The Riemannian curvature $R_y$ is self-adjoint with respect to $\langle\cdot,\cdot\rangle_y$.

Using the Riemannian curvature, one can generalize the sectional curvature  to Finsler
geometry, which is called the flag curvature. Let $y$ be a nonzero tangent vector in
$T_x M$ and $\mathbf{P}$ a $2$-dimensional subspace (a tangent plane) containing $y$.
The pair $(y, \mathbf{P})$ is called a flag and $y$ is called the flag pole.
If $\mathbf{P}$ is linearly spanned by $y$ and $v$, then the flag curvature of the flag
for $(y,\mathbf{P})$ is defined by
\begin{equation}
K(x,y,\mathbf{P})=\frac{\langle R_y v,v\rangle_y}
{\langle y,y\rangle_y \langle v,v\rangle_y-\langle y,v\rangle_y^2}.
\end{equation}
The flag curvature may also be denoted as $K(x,y, y\wedge v)$, or $K^F(\cdot,\cdot,\cdot)$ when the metric needs to be specified. It is obvious that
it does not depends on the choice of the nonzero tangent vector $v$ in $\mathbf{P}$.

It is an important observation of Z.Shen that the Riemannian curvature and flag curvature have a very close relation with relevant curvatures
in Riemannian geometry. In fact,
let $Y$ be a tangent field on an open set $\mathcal{U}\subset M$ which
is nowhere vanishing. Then the Hessian matrices $(g_{ij}(Y(x)))$ define a
smooth Riemannian metric $g_Y$ on $\mathcal{U}$.
We say that $Y$ is a geodesic field on an open subset $\mathcal{U}\subset M$, if the integration
curves of $Y$ are geodesics of nonzero constant speed. Now we have   the following theorem.
\begin{theorem}\label{shen-thm}
{\rm (Shen \cite{CS04})}\quad Let $Y$ be a geodesic field on an open set $\mathcal{U}\subset M$ and  suppose for $x\in \mathcal{U}$,
we have $y=Y(x)\neq 0$. Then the Riemannian curvature $R_y^F$ of $F$ coincides with the Riemannian curvature
${R}_y^{g_Y}$ of the Riemannian metric $g_Y$.
\end{theorem}

It follows from the above theorem that, if $\mathbf{P}$ is a $2$-dimensional tangent plane in $T_x M$ containing $y$, then
$$K^F(x,y,\mathbf{P})={K}^{g_Y}(x,\mathbf{P}).$$

A submanifold $N$ of a Finsler space $(M,F)$ can be naturally endowed a submanifold metric, denoted as $F|_N$. At each point $x\in N$, the Minkowski norm $F|_N(x,\cdot)$
is just the restriction of the Minkowski norm $F(x,\cdot)$ to $T_x N$. We say that $(N,F|_N)$
is a {\it Finsler submanifold} or a {\it Finsler subspace}.

A Finsler submanifold $(N,F|_N)$ of $(M,F)$ is called {\it totally geodesic} if any geodesic of
$(N,F|_N)$ is also a geodesic of $(M,F)$. On a standard local coordinate system $(x^i,y^j)$
such that $N$ is locally defined  by $x^{k+1}=\cdots=x^n=0$, the totally geodesic condition can be equivalently given as
\begin{equation*}
G^i(x,y)=0, \quad k<i\leq n, x\in N, y\in T_x N.
\end{equation*}
A direct calculation shows that in this case the Riemannian curvature $R_y^{F|_N}:T_x N\rightarrow T_x N$ of $(N,F|_N)$ is just the restriction of the Riemannian curvature $R_y^F$ of $(M,F)$, where $y$ is a nonzero tangent
vector of $N$ at $x\in N$. Therefore we have

\begin{proposition}\label{prop-2-2}
Let $(N,F|_N)$ be a totally geodesic submanifold of $(M,F)$. Then for any $x\in N$,
 $y\in T_x N\backslash 0$, and a tangent plane $\mathbf{P}\subset T_x N$, we have
\begin{equation}
K^{F|_N}(x,y,\mathbf{P})=K^F(x,y,\mathbf{P}).
\end{equation}
\end{proposition}
\subsection{Homogeneous Finsler spaces}

A Finsler manifold $(M,F)$ is called  homogeneous if
the full group $I(M,F)$ of isometries   acts transitively on $M$. It is shown in \cite{DengHou2004} that
$G=I(M,F)$ is a Lie transformation group. Let $H$ be the  isotropy subgroup
at a point $o\in M$. Then $M$ is diffeomorphic to the smooth coset space
$G/H$. The tangent space $T_o M$ can be naturally identified with the quotient space
$\mathfrak{m}=\mathfrak{g}/\mathfrak{h}$ through the  natural $\mathrm{Ad}(H)$-actions (the isotropic representation). In many occasions, $\mathfrak{m}$ can be
regarded as a subspace of $\mathfrak{g}$, if we have a direct decomposition
$\mathfrak{g}=\mathfrak{h}+\mathfrak{m}$ of $\mathfrak{g}$ into the sum of invariant  subspaces of the isotropic representation of $H$.

Similarly,  for any closed subgroup $G$ of $I(M,F)$ which acts transitively on $M$,
we have a presentation $M=G/H$, where $H$ is the isotropic subgroup of $G$ at a fixed point. Sometimes, we may also drop the requirement that
$G\subset I(M,F)$ (or $I_0(M,F)$ when $G$ is connected), and only assume that $G$ acts
on $(M,F)$ isometrically. For the same manifold, different homogeneous presentations
may reveal different geometric properties of the manifold. The most typical example is the nine
classes of homogeneous spheres \cite{BO40}. But sometimes they do not affect the
(local) geometric properties of manifolds. So, to study homogeneous Finsler geometry,
it is convenient to introduce an equivalence relation among different homogeneous spaces.
We will make this equivalence relation precise for normal homogeneous spaces below.

By the homogeneity,
the geometric quantities such as  curvatures of a homogeneous Finsler space
can be reduced to an $\mathrm{Ad}(H)$-invariant vectors in certain tensor algebra of $\mathfrak{m}$.
On the other hand, the homogeneous metric $F$ itself is completely determined by
an $\mathrm{Ad}(H)$-invariant Minkowski norm on $\mathfrak{m}$; see \cite{DE12}.

\subsection{Submersion and subduced metric}
Before defining normal homogeneous spaces in Finsler geometry,
we first briefly review the theory of Finslerian submersions. For details we refer to \cite{PD01}.

A linear map $\pi: (\mathbf{V}_1,F_1)\rightarrow (\mathbf{V}_2,F_2)$ between two Minkowski
spaces is called an {\it isometric submersion} (or simply {\it submersion}), if it maps
the unit ball $\{y\in\mathbf{V}_1|F_1(y)\leq 1\}$ in $\mathbf{V}_1$ onto the unit ball
$\{y\in\mathbf{V}_2| F_2(y)\leq 1\}$ in $\mathbf{V}_2$. It is obvious that a submersion map
$\pi$ must be surjective, and that the Minkowski norm $F_2$ on $\mathbf{V}_2$ is uniquely determined by the following equality:
\begin{equation}\label{0001}
F_2(w)=\inf\{F_1(v)|\pi(v)=w\}.
\end{equation}
Given a Minkowski space $(\mathbf{V}_1,F_1)$ and a surjective linear map
$\pi:\mathbf{V}_1\rightarrow\mathbf{V}_2$, there exists a unique Minkowski norm
$F_2$ on $\mathbf{V}_2$ such that $\pi$ is a submersion. We usually say that $F_2$
is the subduced norm through $\pi$; see \cite{PD01}.

Now we clarify  the relationship between the Hessian matrices of $F_1$ and $F_2$. For this we need the notion
of horizonal lift of vectors. Given a nonzero vector $w$ in $\mathbf{V}_2$, the infimum in (\ref{0001})
can be reached at a unique vector $v\in \mathbf{V}_1$. We call $v$   the horizonal lift of $w$ with respect to the submersion $\pi$.
Obviously the horizontal lift of $0$ is the zero vector.
The horizonal lift $v$ can also be determined by
\begin{equation}
\langle w,\ker \pi\rangle_v=0.
\end{equation}
Now the Hessian matrix of $F_2$ at $w$ is determined by the following proposition.
\begin{proposition}\label{prop-1}
Let $\pi:(\mathbf{V}_1,F_1)\rightarrow(\mathbf{V}_2,F_2)$ be a submersion between Minkowski
spaces. Assume that $v$ is the horizonal lift of the nonzero vector $w$ in $\mathbf{V}_2$. Then $\pi:(\mathbf{V}_1,\langle\cdot,\cdot\rangle_v^{F_1})\rightarrow (\mathbf{V}_2,\langle\cdot,\cdot\rangle_w^{F_2})$ is
also a submersion between Euclidean spaces.
\end{proposition}


A submersion between two Finsler spaces $(M_1, F_2)$ and $(M_2, F_2)$ is a surjective map
$\rho$ such that for any $x\in M_1$, the tangent map
$d\rho|x: (T_x(M_1), F)\to (T_{\rho(x)}(M_2), F_2)$ is a submersion between Minkowski spaces. The horizonal lift of a tangent vector field
can be similarly defined for a submersion between Finsler spaces.
Then the corresponding integration curves define the horizonal lift of smooth curves.
Horizonal lift provides a one-to-one correspondence between the geodesics on $M_2$
and the horizonal geodesics on $M_1$, so the horizonal lift of a geodesic field
is also a geodesic field. Now Theorem \ref{shen-thm}, Proposition \ref{prop-1} and the curvature formula for
Riemannian submersions give the following theorem in \cite{PD01}.
\begin{theorem}\label{theorem-3-2}
Let $\rho:(M_1,F_1)\rightarrow (M_2,F_2)$ be a submersion. Assume that $x_2=\rho(x_1)$,
and $y_2,v_2\in T_{x_2}M_2$  are linearly independent. Let $y_1$ be the horizonal lift
of $y_2$, and $v_1$ the horizonal lift of $v_2$ for the induced submersion
$\rho_*:(T_{x_1}M_1,\langle\cdot,\cdot\rangle_{y_1}^{F_1})\rightarrow
(T_{x_2}M_2,\langle\cdot,\cdot\rangle_{y_2}^{F_2})$. Then we have
\begin{equation}
{K}^{F_1}(x_1,y_1,y_1\wedge v_1)\leq
K^{F_2}(x_2,y_2,y_2\wedge v_2).
\end{equation}
\end{theorem}

Submersions can be similarly defined for vector spaces with singular Minkowski norms and
smooth manifolds with singular Finsler metrics. In the level of norms in linear spaces, the singular Minkowski
norm of the domain space can still induce a unique singular Minkowski norm of the target space. But in general horizonal lift can not be uniquely defined, so it does not share most of the properties for the submersion between smooth Minkowski norms or Finsler manifolds.

\section{Normal homogeneous Finsler spaces}
In this section, we will define the notion of normal homogeneous Finsler spaces and study their fundamental geometric properties. Unless otherwise stated, Lie groups and smooth manifolds will always be assumed to be connected.

\subsection{Definition of normal homogeneous spaces}

Let $G$ be a connected Lie group with a bi-invariant Finsler metric $\bar{F}$,
and $H$  a closed
subgroup of $G$. Denote the Lie algebras of $G$ and $H$  as $\mathfrak{g}$ and $\mathfrak{h}$,
respectively. Let $\rho$ be the natural projection from $G$ to $M=G/H$.

The existence of
a bi-invariant Finsler metric on $G$  implies that the universal covering of $G$ is a product of a compact
simply connected Lie group and a Euclidean space. Thus there exists  a bi-invariant inner product
on $\mathfrak{g}$. Fixing a bi-invariant inner product on $\mathfrak{g}$, we then   have the orthogonal decomposition
$\mathfrak{g}=\mathfrak{h}+\mathfrak{m}$ such that
$[\mathfrak{h},\mathfrak{m}]\subset\mathfrak{m}$. Moreover,
the linear space $\mathfrak{m}$
can be identified with  the tangent space of $M$ at the point $o=\rho(e)$.

The following lemma gives the existence of the subduced metric in this situation.
\begin{lemma}\label{lemma-3-1}
Keep all the notations as above. There exists a uniquely defined $G$-invariant metric $F$
on $M$ such that for any  $g\in G$, the tangent map $\rho_*:(T_g G,\bar{F}(g,\cdot))\rightarrow
(T_{\pi(g)}M,F(\pi(g),\cdot))$ is a submersion.
\end{lemma}

\begin{proof}
The tangent map $\rho_*:T_e G\rightarrow T_o M$ defines a unique subduced Minkowski norm
$F$ on $\mathfrak{m}=T_o M$ from $\bar{F}(e,\cdot)$. Since $\bar{F}(e,\cdot)$ is
$\mathrm{Ad}(H)$-invariant, $F$ is also $\mathrm{Ad}(H)$-invariant. Then there exists a (unique) $G$-invariant Finsler metric on $M$ whose restriction to $T_oM$ is equal to
$F$ (see \cite{DE12}).
For simplicity, we  denote this Finsler metric as $F$.
Now given $g\in G$, the tangent map $\rho_*|_{T_g G}=g_*\circ\pi_*|_{T_e G}\circ (L_{g^{-1}})_*$
is a submersion between the Minkowski spaces  $(T_g G,\bar{F}(g,\cdot))$ and $(T_{\pi(g)}M,F(\pi(g),\cdot))$,
since $\rho_*|_{T_e G}$ is a submersion, and $\bar{F}$ is bi-invariant.
\end{proof}

The $G$-invariant Finsler metric $F$ defined in Lemma \ref{lemma-3-1} will be called the {\it normal
homogeneous metric} induced by $\bar{F}$. We will also call $(M,F)$ the normal homogeneous space
induced by $\rho:G\rightarrow M=G/H$ and $\bar{F}$.

This definition is a natural generalization of the Riemannian normal homogeneous space.
Although in most cases the normal homogeneity is only referred to a compact coset space $M=G/H$  of a  compact Lie group $G$, the theory can be applied to some special non-compact coset spaces with  very minor technical adjustment. Note that if $G$ is a compact simple Lie group,
then there exists a unique bi-invariant Riemannian metric on $G$ (up to homotheties).
  However, there usually exists infinitely many bi-invariant Finsler metrics  on a compact Lie group (up to homotheties),  even for a compact simple Lie group.
This implies  that  normal homogeneous Finsler metrics are not     related to the group $G$ as closely as   in the Riemannian case. Hence   the problem
of this paper is much more difficult than the same problem in the Riemannian case.

We end this subsection with the following easy but useful observation on normal metrics on  coset spaces of products of Lie groups.

Let $\bar{F}_1$ be a bi-invariant Finsler metric on $G_1=G_2\times G_3$. Then it induces
a bi-invariant normal homogeneous metric $\bar{F}_2$ on $G_2$ through the projection
from $G_1$ to its $G_2$-factor. For any closed subgroup $H_2\in G_2$, denote $H_1=H_2\times G_3$. Then $\bar{F}_1$ and $\bar{F}_2$ define isometric normal homogeneous metrics $F_1$ and
$F_2$ respectively on $M_1=G_1/H_1=G_2/H_2=M_2$. Any normal homogeneous metric $F_2$ on $G_2/H_2$ can be obtained in this way.
The identity map gives a many-to-one correspondence between
the normal homogeneous Finsler metrics on $G_1/H_1$
and the normal homogeneous Finsler metrics on $G_2/H_2$.

\subsection{The flag curvature and S-curvature}
Now we consider the fundamental geometric properties of normal homogeneous Finsler spaces. It turns out that this class of spaces behave very well.
We first prove
\begin{proposition}\label{prop-3-2}
A normal homogeneous space has vanishing S-curvature and non-negative flag curvature.
\end{proposition}
\begin{proof}
Let  $(M,F)$ be a normal homogeneous Finsler space  induced by the projection $\rho:G\rightarrow M=G/H$
and the bi-invariant metric $\bar{F}$ on $G$. Then   $\bar{F}$ is a Berwald metric (see \cite{DE12}), i.e.,
its Chern connections coincide with the Levi-Civita connection of a Riemannian metric. So the flag curvature of $(G,\bar{F})$ is non-negative.
Moreover, it is also known that any left or right invariant field on $G$ is a geodesic field.
Now the statement for the flag curvature follows from this observation and Theorem \ref{theorem-3-2}.

 Now we prove that $(M, F)$ has vanishing S-curvature.
On a   local standard coordinate system, the S-curvature $S(x,y)$ is the derivative of the distortion function
\begin{equation}
\tau(x,y)=\ln\frac{\sqrt{\det(g_{ij}(x,y))}}{\sigma(x)}
\end{equation}
in the direction of the geodesic spray $G(x,y)$.
By the homogeneity, we only need to prove that
$S(o,y)=0$ for any nonzero tangent vector $y\in \mathfrak{m}=T_o M$. Let $\bar{y}\in T_e G=
\mathfrak{g}$ be the horizonal lift of $y$. The one-parameter subgroup
$\bar{c}(t)=\exp (t\bar{y})$
is a horizonal geodesic which is the horizonal lift of the geodesic $c(t)$ at $o\in M$ in the direction of $y$. The vector $\bar{y}\in\mathfrak{g}$ defines a right invariant vector field
on $G$ which induces a Killing vector field $Y$ on $M$. The integration curve of $Y$ at $o$
coincides with the geodesic $c(t)$. So along $(c(t),\dot{c}(t))$, the distortion function
$\tau$ is a constant function. Therefore  $S(o,y)$ is equal to $0$.
\end{proof}

As S-curvature has some mysterious relationship  with other curvatures in Finsler geometry,
 Proposition \ref{prop-3-2} may be useful in the  study of other curvatures  of normal homogeneous spaces.

Normal homogeneous Finsler spaces provide us with a large
class of complete Finsler manifolds of non-negative flag curvature. Thus it is important to find out the condition for a normal homogeneous Finsler space to have strictly positive flag curvature. For simplicity,
a Finsler space with strictly positive flag curvature will be simply called {\it positively curved}.
Recall that positively curved normal homogeneous Riemannian manifolds have been classified by Berger in \cite{Ber61}, and this classification produces several new examples of positively
 curved homogeneous Riemannian manifolds besides the  rank one symmetric spaces.
To find  all possible (connected) coset spaces admitting positively curved normal homogeneous Finsler metrics, we will need more accurate results on the flag curvature, which will be discussed in the next subsection.

\subsection{Flat splitting subalgebras and vanishing of flag curvature}
Let $F$ be a normal homogeneous Finsler metric on $M=G/H$ induced by the projection
$\rho:\mathfrak{g}\rightarrow\mathfrak{h}$ and a bi-invariant Finsler metric $\bar{F}$
on $G$. We keep all the notations as above.

Let $\mathfrak{g}=\mathfrak{h}+\mathfrak{m}$ be the orthogonal decomposition
with respect to a fixed bi-invariant inner product on $\mathfrak{g}$.
A commutative subalgebra $\mathfrak{s}$ of $\mathfrak{g}$ is called a
{\it flat splitting subalgebra} with respect to the pair $(\mathfrak{g},\mathfrak{h})$ (or with respect to the decomposition $\mathfrak{g}=\mathfrak{h}+\mathfrak{m}$) if the following conditions are satisfied
\begin{description}
\item{\rm (1)} $\mathfrak{s}$ is the intersection of a family of Cartan subalgebras
$\mathfrak{t}_a$, $a\in\mathcal{I}$, i.e., $\mathfrak{s}=\bigcap_{a\in\mathcal{I}}\mathfrak{t}_a$;
\item{\rm (2)} $\mathfrak{s}$ is splitting,  that is, $\mathfrak{s}=(\mathfrak{s}\cap\mathfrak{h})+(\mathfrak{s}\cap\mathfrak{m})$;
\item{\rm (3)} $\dim\mathfrak{s}\cap\mathfrak{m}>1$.
\end{description}
If furthermore $\mathfrak{s}$ is a Cartan subalgebra of $\mathfrak{g}$, then it is called
a {\it flat splitting Cartan subalgebra} (FSCS for simplicity).

The following theorem is the main technique to reduce the classification of  connected
coset spaces admitting positively curved normal homogeneous metrics to an algebraic problem.

\begin{theorem}\label{thm-4-2}
Keep all  the notations as above and assume that the normal homogeneous Finsler metric $F$ on
$M=G/H$ is positively curved. Then for any closed subgroup $K$ of $G$ containing the identity
component $H_0$ of $H$ and the corresponding orthogonal decomposition $\mathfrak{g}=\mathfrak{k}+\mathfrak{p}$ with respect to the same fixed bi-invariant
inner product, there does not exist any FSCS or  flat splitting subalgebra of $\mathfrak{g}$ with respect to the pair $(\mathfrak{g},\mathfrak{k})$.

In particular, if $\dim K<\dim G$, then $G/K$ admits positively curved normal homogeneous
Finsler metrics. Moreover,   there dose not exist any FSCS or flat splitting subalgebra of
$\mathfrak{g}$ with respect to the decomposition $\mathfrak{g}=\mathfrak{h}+\mathfrak{m}$.
\end{theorem}

Theorem \ref{thm-4-2} makes sense for the term ``flat'' in these newly defined notations.
Compared to the classification of positively curved normal homogeneous Riemannian spaces, there exists a more intrinsic connection between the condition that there is flat
splitting subalgebra or FSCS, and the condition  that there is a commuting linearly independent pair in $\mathfrak{m}$ (we will simply refer to this as  Berger's condition because it first appears in Berger's classification work). In all the cases we will consider in this paper, FSCS or flat splitting subalgebra can be constructed using the  Berger's condition. But we can not find a proof for this fact without a case by case discussion.

For any $K$ in Theorem \ref{thm-4-2}, there is a natural projection $\pi:G/H_0\rightarrow G/K$. With respect to the normal homogeneous Finsler metric on $G/K$ induced by the same bi-invariant Finsler metric $\bar{F}$ as for $M=G/H$, this projection is also a submersion. We assert that the normal homogeneous space $G/H_0$ induced by the same metric
$\bar{F}$  is also positively curved. In fact, By Theorem \ref{theorem-3-2}, if $\dim G/K>1$, i.e.,
if $\dim G -\dim K >1$, then $G/K$
is positively curved with the normal homogeneous Finsler metric induced by the same $\bar{F}$ on $G$. On the other hand, the case
$\dim G-\dim K=1$ can not happen, since otherwise the positively curved manifold $G/H_0$ must be compact, hence  $G/K$ must be a circle. Then $G/H_0$ has an
infinite fundamental group,  which is a contradiction. This observation reduces
the proof of Theorem \ref{thm-4-2} to the case $K=H$.

\subsection{Proof of Theorem \ref{thm-4-2}}
As remarked in the above subsection, we only need to prove Theorem \ref{thm-4-2} for $K=H$.

Assume conversely that
there is a flat splitting subalgebra $\mathfrak{s}=\bigcap_{a\in \mathcal{I}}\mathfrak{t}_a$
for the family of Cartan subalgebras $\mathfrak{t}_a$, $a\in \mathcal{I}$. It is obvious that the set $\mathcal{I}$ can be chosen to be finite. For each $a\in \mathcal{I}$,
there is a a closed torus $T_a$ with $\mathrm{Lie}(T_a)=\mathfrak{t}_a$. Then the identity component $S$ of $\bigcap_{a\in \mathcal{I}}T_a$ is a closed
connected subgroup of $G$ with $\mathrm{Lie}(S)=\mathfrak{s}$. The orbit $S\cdot o$ is
a closed torus in $M$ with $\dim S\cdot o>1$.
Then the Finsler metric $F|_{S\cdot o}$ on the submanifold $S\cdot o$ is locally Minkowski, which has constant
$0$ flag curvature.

We now prove that $(S\cdot o, F|_{S\cdot o})$ is a totally geodesic submanifold in $(M,F)$. We need the following
lemma.

\begin{lemma}\label{lemma-4-3} Let $\rho_*: (\mathfrak{g},\langle\cdot,\cdot\rangle^{\bar{F}})\rightarrow
(\mathfrak{m},\langle\cdot,\cdot\rangle^F)$ be a submersion. Then the horizonal lift of
any vector of $\mathfrak{s}\cap\mathfrak{m}$ must be contained in $\mathfrak{s}$.
\end{lemma}
\begin{proof}
Since $\bar{F}(e,\cdot)$ is $\mathrm{Ad}(G)$-invariant, for any $a\in \mathcal{I}$, and at any
$y\in\mathfrak{t}_a\backslash\{0\}$, the derivative of $\bar{F}$ vanishes in the directions of $\mathfrak{t}_a^{\perp}$ (with respect to the chosen bi-invariant
inner product on $\mathfrak{g}$).
So at any nonzero vector $y\in\mathfrak{s}=\bigcap_{a\in\mathcal{I}}\mathfrak{t}_a$,
$\bar{F}$ vanishes in all the directions of
$\mathfrak{s}^\perp=\sum_{a\in\mathcal{I}}
\mathfrak{t}_a^{\perp}$.
Now we consider the restriction of the Minkowski norm $\bar{F}(e,\cdot)$ to
$\mathfrak{s}$. For any nonzero vector $y$ in $\mathfrak{s}$,
there exists a unique $y'\in y+\mathfrak{s}\cap\mathfrak{h}\subset \mathfrak{s}$, such
that the derivative of $\bar{F}(e,\cdot)$ vanishes in the directions of
$\mathfrak{s}\cap\mathfrak{h}$. Then at $y'$, the derivative of $\bar{F}(e,\cdot)$ vanishes
in all the directions of $\mathfrak{s}\cap\mathfrak{h}
+\mathfrak{s}^\perp$, since $\mathfrak{s}$ is splitting, and
\begin{equation}
\mathfrak{h}=\mathfrak{s}\cap\mathfrak{h}+
\mathfrak{s}^\perp\cap\mathfrak{h}\subset
\mathfrak{s}\cap\mathfrak{h}
+\mathfrak{s}^\perp.
\end{equation}
Therefore  $y'$ is the
horizonal lift of $y$ with respect to the submersion
$\rho_*:(\mathfrak{g},\bar{F}(e,\cdot))\rightarrow (\mathfrak{m},F(o,\cdot))$.
\end{proof}

For any nonzero tangent vector
$y\in T_o (S\cdot o)=\mathfrak{s}\cap\mathfrak{m}$, the curve $\exp (ty)\cdot o$ is the geodesic
of $(S\cdot o, F|_{S\cdot o})$ with initial vector  $y$. With respect to the projection $\rho$, it is covered by $\exp(ty')\cdot o$, which is a horizonal geodesic of $(G,\bar{F})$. So
$\exp(ty)\cdot o$ is also a geodesic of $(M,F)$. This proves that $(S\cdot o,F|_{S\cdot o})$
is totally geodesic in $(M,F)$. Now by Proposition \ref{prop-2-2}, the flag curvature of
$(M,F)$ for any tangent plane of $S\cdot o$ (which does exist due to the dimension condition) vanishes. This implies that $(M,F)$ is not positively curved, which is a contradiction. This completes the proof of Theorem \ref{thm-4-2}.

\begin{remark}
Theorem \ref{thm-4-2} can also be proven using a flag curvature formula in \cite{XDHH2014} for Finslerian submersions.
\end{remark}

\section{Classification of positively curved normal homogeneous spaces}

In this section we will develop the techniques to classify normal homogeneous Finsler spaces with positive flag curvature.
\subsection{Equivalence between normal homogeneous spaces}


As  flag curvature is only relevant to the local geometric properties of a Finsler space, we will
not distinguish normal homogeneous spaces $M_1=G_1/H_1$ and $M_2=G_2/H_2$ which are induced by the bi-invariant Finsler metrics $\bar{F}_i$ on $G_i$, respectively, in the following cases:

\begin{description}
\item{\rm (1)} $G_2$ is a covering group of $G_1$ with $\bar{F}_2$ induced from $\bar{F}_1$, and $H_2$ has the same identity component as $H_1$.
\item{\rm (2)} $G_1=G_2\times G_3$, $H_1=H_2\times G_3$, and there is a bi-invariant metric
$\bar{F}_3$ on $G_3$ and a real function $f$ on $\mathbb{R}^2$ which is  positively homogeneous of degree $1$, such that $\bar{F}_1=f(\bar{F}_2,\bar{F}_3)$.
\item{\rm (3)} There is an isomorphism $\phi$ between $G_1$ and $G_2$ which maps $H_1$ onto $H_2$,
and  $\phi^* (\bar{F}_2)=\bar{F}_1$.
\end{description}

We will call two normal homogeneous spaces $M_1=G_1/H_1$ and $M_2=G_2/H_2$
{\it equivalent}, if
there are a finite sequence of normal homogeneous spaces connecting $M_1$ and $M_2$, such
that each adjacent pair satisfies  the conditions in one of the above three cases.

This equivalence provides a lot of convenience for us. Since
we only need to study the classification of positively
curved Finsler spaces up to equivalence,
we can assume that the Lie group $G$ for the
 normal homogeneous space $M=G/H$ to be simply connected (which means that $G$ may have an Euclidean factor and may be non-compact) and the subgroup $H$ to be connected. Moreover, in some special cases, we can also reduce $G$ and $H$ to their appropriate subgroups, or replace $H$ with its image under an automorphism of $G$. Sometimes this method can  give the normal homogeneous space the most standard and recognizable presentation.

\subsection{Some general classification results}

Keep all the notations as above. We call the dimension of the Cartan subalgebra in $\mathfrak{g}$ the rank of $G$ (or $\mathfrak{g}$), denoted as $\mathrm{rk}G$
(or $\mathrm{rk}\mathfrak{g}$). We will also call the maximal dimension of the commutative  subalgebras in $\mathfrak{m}$ the rank of $M$.

The following two propositions are direct corollaries of  Theorem \ref{thm-4-2}.

\begin{proposition}\label{prop-5-1}
Let $M=G/H$ be a positively curved normal homogeneous space.  Then either $\mathrm{rk}G=\mathrm{rk}H$ or
$\mathrm{rk}G=\mathrm{rk}H+1$.
\end{proposition}

\begin{proposition}\label{prop-5-2}
Let $M=G/H$ be a positively curved symmetric normal homogeneous space. Then $M$ is positively curved if and only if it is of rank $1$.
\end{proposition}

Recall that the normal homogeneous space $G/H$ is symmetric if and only if in the decomposition $\mathfrak{g}=\mathfrak{h}+\mathfrak{m}$, we have $[\mathfrak{m}, \mathfrak{m}]\subseteq\mathfrak{h}$. If $G/H$ is a positively curved normal homogeneous space with $\mathrm{rk}G>\mathrm{rk}H+1$, then we can construct a splitting Cartan
subalgebra of $\mathfrak{g}$ from any Cartan subalgebra of $\mathfrak{h}$, which
must be a FSCS by the assumption $\mathrm{rk}G>
\mathrm{rk}H+1$. This is a contradiction to Theorem \ref{thm-4-2}.
In \cite{XDHH2014}, we present another proof for the same conclusion where we only require $M$
to be a positive curved homogeneous Finsler space.

Now we prove Proposition \ref{prop-5-2}. In fact,   otherwise there exists a maximal commutative subalgebra $\mathfrak{a}$
in $\mathfrak{m}$ with $\dim\mathfrak{a}>1$. Then the symmetric condition implies that any Cartan subalgebra of $\mathfrak{g}$ containing $\mathfrak{a}$ is a FSCS, which is a contradiction.
Proposition \ref{prop-5-2} can also be proven by arguing that
$(M,F)$ is Berwald with the same connection and Riemannian curvature as a Riemannian
normal homogeneous space; see \cite{DE12}.

The positively curved normal homogeneous Finsler spaces with $\mathrm{rk}G=\mathrm{rk}H$ and $\mathrm{rk}G=\mathrm{rk}H+1$ will be
discussed  in the next two sections.

\subsection{Some notations and techniques}
We now summarize some basic facts on compact Lie groups from \cite{HE78}, and develop some techniques for our classification.

Assume that the normal homogeneous space $M=G/H$ is connected, simply connected and positively curved, $G$ is connected and simply connected and $H$ is connected. Fix
an orthogonal decomposition $\mathfrak{g}=\mathfrak{h}+\mathfrak{m}$ with respect
to a given bi-invariant inner product on $\mathfrak{g}$.

Fix a Cartan subalgebra $\mathfrak{t}$ such that $\mathfrak{t}\cap\mathfrak{h}$ has
 maximal dimension, i.e., $\mathfrak{t}\subset\mathfrak{h}$ if $\mathrm{rk}G=\mathrm{rk}H$, or $\mathfrak{t}\cap\mathfrak{h}$ is a subspace of $\mathfrak{t}$ with  codimension $1$ if $\mathrm{rk}G=\mathrm{rk}H+1$.
In the following, unless otherwise stated, the notations  (e.g.,  roots, root systems, root planes, etc.) will be defined with respect to
$\mathfrak{t}$ when concerning $\mathfrak{g}$ or its subalgebras containing $\mathfrak{t}$, or to $\mathfrak{t}\cap\mathfrak{h}$ when concerning $\mathfrak{h}$.

With respect to $\mathfrak{t}$, we have a decomposition
\begin{equation}
\mathfrak{g}=\mathfrak{t}+\sum_{\alpha\in\Delta^+_\mathfrak{g}}\mathfrak{g}_{\pm\alpha}.
\end{equation}
The $2$-dimensional subspace $\mathfrak{g}_{\pm\alpha}$ is called the root plane of
$\mathfrak{g}$. Through
the given bi-invariant inner product, any root $\alpha$ can be naturally identified with a vector in $\mathfrak{t}$, which generates
$[\mathfrak{g}_{\pm\alpha},\mathfrak{g}_{\pm\alpha}]$.
The bracket relation between different root planes is given by the following well known formula,
\begin{equation}\label{0003}
[\mathfrak{g}_{\pm\alpha},\mathfrak{g}_{\pm\beta}]=\mathfrak{g}_{\pm(\alpha+\beta)}+
\mathfrak{g}_{\pm(\alpha-\beta)},
\end{equation}
where $\mathfrak{g}_{\pm\alpha}$ and $\mathfrak{g}_{\pm\beta}$ are different root planes, i.e. $\alpha$ and $\beta$ are linearly independent, and
each term of the right side can be $0$ when the corresponding vectors are not roots of $\mathfrak{g}$.

We now assert that if the right side of (\ref{0003}) is $2$-dimensional,
then for any nonzero vector $v$ of $\mathfrak{g}_{\pm\alpha}$, $\mathrm{ad}(v)$ maps
$\mathfrak{g}_{\pm\beta}$ isomorphically onto the right side. In fact, assume that $\alpha+\beta$ is a root but $\alpha-\beta$ is not. Then for any $u\in (\alpha+\beta)^{\perp}$ which is not perpendicular to either $\alpha$ or $\beta$, we have
\begin{equation}
[[u,v],\mathfrak{g}_{\pm\beta}]=[v,[u,\mathfrak{g}_{\pm\beta}]]=[v,\mathfrak{g}_{\pm\beta}].
\end{equation}
As $\mathfrak{g}_{\pm\alpha}$ is linearly generated by $v$ and $[u,v]$,
$[v,\mathfrak{g}_{\pm\beta}]=[\mathfrak{g}_{\pm\alpha},\mathfrak{g}_{\pm\beta}]$ is two
dimensional, that is, $\mathrm{ad}(v)$ is an isomorphism from $\mathfrak{g}_{\pm\beta}$ onto
$\mathfrak{g}_{\pm(\alpha+\beta)}$.

The above discussion can be summarized as the following lemma.
\begin{lemma}
Keep all the notations as above.
\begin{description}\label{trick-lemma-0}
\item{\rm (1)} For any root $\alpha$ of $\mathfrak{g}$, $[\mathfrak{g}_{\pm\alpha},
\mathfrak{g}_{\pm\alpha}]$ is the line in $\mathfrak{t}$ spanned by the vector which is dual to $\alpha$ with respect to the chosen bi-invariant inner product.
\item{\rm (2)} Let $\alpha$ and $\beta$ be two linearly independent roots  of $\mathfrak{g}$. If neither $\alpha+\beta$
nor $\alpha-\beta$ is a  root of $\mathfrak{g}$, then
$[\mathfrak{g}_{\pm\alpha},\mathfrak{g}_{\pm\beta}]$ is equal to $0$; If $\alpha+\beta$ (resp. $\alpha-\beta)$) is a root, but the other is not,  then $[\mathfrak{g}_{\pm\alpha},\mathfrak{g}_{\pm\beta}]$ is equal to the $2$-dimensional root plane
$\mathfrak{g}_{\pm(\alpha+\beta)}$ (resp. $\mathfrak{g}_{\pm(\alpha-\beta)}$); If both $\alpha\pm\beta$ are  roots of $\mathfrak{g}$, then $[\mathfrak{g}_{\pm\alpha},\mathfrak{g}_{\pm\beta}]$ is equal to
  the $4$-dimensional sum $\mathfrak{g}_{\pm(\alpha+\beta)}+\mathfrak{g}_{\pm(\alpha-\beta)}$.
\item{\rm (3)} In (2), if $[\mathfrak{g}_{\pm\alpha},\mathfrak{g}_{\pm\beta}]$ is
$2$-dimensional, then for any nonzero vector $v\in\mathfrak{g}_{\pm\alpha}$, the linear map $\mathrm{ad}(v)$
 is an isomorphism from $\mathfrak{g}_{\pm\beta}$  onto $[\mathfrak{g}_{\pm\alpha},\mathfrak{g}_{\pm\beta}]$.
\end{description}
\end{lemma}

When $\mathrm{rk}G=\mathrm{rk}H$, the roots and root planes of $\mathfrak{h}$ are also those
of $\mathfrak{g}$. Any root plane $\mathfrak{g}_{\pm\alpha}$ of
$\mathfrak{g}$ is either contained in $\mathfrak{h}$ (when $\alpha$ is a root of $\mathfrak{h}$), or contained in $\mathfrak{m}$ (when $\alpha$ is not a root of $\mathfrak{h}$). For simplicity, in the later case we will say that $\alpha$ is a root of $\mathfrak{m}$.

When $\mathrm{rk}G=\mathrm{rk}H+1$, the roots of $\mathfrak{h}$ can be naturally
identified with vectors on $\mathfrak{t}\cap\mathfrak{h}$ through the restriction of the given bi-invariant
inner product to $\mathfrak{t}\cap\mathfrak{h}$. Let $\mathrm{pr}$ be the orthogonal projection from $\mathfrak{t}$ to
 $\mathfrak{t}\cap\mathfrak{h}$. Then for any root $\alpha'\in\Delta^+_\mathfrak{h}$, the root plane
$\mathfrak{h}_{\pm\alpha'}$ is contained in
\begin{equation}
\hat{\mathfrak{g}}_{\pm\alpha'}=\sum_{\mathrm{pr}(\alpha)=\alpha'}\mathfrak{g}_{\pm\alpha}.
\end{equation}
In particular, if $\mathrm{pr}(\alpha)$ is not a root of $\mathfrak{h}$ for the root $\alpha$
of $\mathfrak{g}$, then $\mathfrak{g}_{\pm\alpha}\subset\mathfrak{m}$. In this case,   we  say that $\pm\alpha$ are roots of $\mathfrak{m}$.

The following lemmas will be used repeatedly in our later discussion.
\begin{lemma}\label{trick-lemma-1}
Assume that $M=G/H$ is a positively curved normal homogeneous Finsler space with
$\mathrm{rk}G=\mathrm{rk}H$.
\begin{description}
\item{\rm (1)} Let   $\alpha$ and $\beta$ be two roots of $\mathfrak{m}$ with $\alpha\neq\pm\beta$. Then either $\alpha+\beta$
or $\alpha-\beta$ is a root of $\mathfrak{g}$.
\item{\rm (2)} Let $\alpha$ and $\beta$ be two roots of $\mathfrak{m}$ with   angle  $\frac{\pi}{3}$
or $\frac{2\pi}{3}$, and suppose $\alpha$ and $\beta$ do not belong to a $G_2$-factor of $\mathfrak{g}$.
Then
either $\alpha+\beta$ or $\alpha-\beta$ is a root of $\mathfrak{h}$.
\end{description}
\end{lemma}

\begin{proof}
(1)\quad If $\alpha\neq\pm\beta$ and $\alpha\pm\beta$ are not roots of $\mathfrak{g}$, then for any nonzero vectors $v_1\in\mathfrak{g}_{\pm\alpha}$ and
$v_2\in\mathfrak{g}_{\pm\beta}$, we have $[v_1, v_2]=0$. Let $\alpha^\perp$, resp. $\beta^\perp$, be   the orthogonal complement of $\alpha$, resp. $\beta$,  in $\mathfrak{t}$ with respect to the chosen bi-invariant inner product. Then the linear span of $v_1, v_2$ and $\alpha^\perp\cap\beta^\perp$ is   a FSCS,
which is a contradiction to Theorem \ref{thm-4-2}.

(2)\quad Suppose conversely that $\pm\alpha\pm\beta$ are roots of $\mathfrak{m}$.
Then the subalgebra $\mathfrak{g}'=\mathbb{R}\alpha+\mathbb{R}\beta+\sum_{a,b\in\mathbb{Z}}
\mathfrak{g}_{\pm(a\alpha+b\beta)}$ of $\mathfrak{g}$ is isomorphic to $\mathfrak{su}(3)$,
by an isomorphism $l$ which maps $\mathfrak{g}'\cap\mathfrak{h}=\mathbb{R}\alpha+
\mathbb{R}\beta$ to the diagonal matrices.
Now we consider the  matrices
\begin{equation*}\label{2-matrices}
u=\left(  \begin{array}{ccc}
    0 & 1 & 1 \\
    -1 & 0 & 1 \\
    -1 & -1 & 0 \\
  \end{array}
\right),\mbox{ and }\quad v=
\left(  \begin{array}{ccc}
    0 & \sqrt{-1} & -\sqrt{-1} \\
    \sqrt{-1} & 0 & \sqrt{-1} \\
    -\sqrt{-1} & \sqrt{-1} & 0 \\
  \end{array}
\right)\mbox{ in }\mathfrak{su}(3).
\end{equation*}
Then $l^{-1}(u)$ and $l^{-1}v$ are  linearly independent and  commutative. Thus the linear span of $l^{-1}(u)$, $l^{-1}(v)$ and
$\alpha^\perp\cap\beta^\perp$ is a FSCS. This is a contradiction to Theorem \ref{thm-4-2}.
\end{proof}

\begin{lemma}\label{trick-lemma-2}
Let $M=G/H$ be a positively curved normal homogeneous space with
$\mathrm{rk}G=\mathrm{rk}H+1$.
\begin{description}
\item{\rm (1)} If there is a root $\alpha$ of $\mathfrak{g}$ perpendicular to
$\mathfrak{t}\cap\mathfrak{h}$, then $\mathfrak{g}_{\pm\alpha}\subset\mathfrak{m}$.
\item{\rm (2)} Any root $\alpha\in\mathfrak{t}\cap\mathfrak{h}$ of $\mathfrak{g}$  is a root of $\mathfrak{h}$.
\item{\rm (3)} Let   $\alpha$ and $\beta$ be two roots of $\mathfrak{m}$ with angle $\frac{\pi}{3}$
or $\frac{2\pi}{3}$. Suppose $\mathfrak{t}\cap\mathfrak{h}^{\perp}\subset\mathbb{R}\alpha+\mathbb{R}\beta$ and
$\alpha$ and $\beta$ do not belong to a $G_2$-factor of $\mathfrak{g}$.
Then
either $\mathrm{pr}(\alpha+\beta)$ or $\mathrm{pr}(\alpha-\beta)$
is a root of $\mathfrak{h}$.
\end{description}
\end{lemma}

\begin{proof}
(1)\quad This follows directly from the above observations and the fact that
  $\mathrm{pr}(\alpha)=0$, which is not a root of $\mathfrak{h}$.

(2)\quad If $\alpha=\mathrm{pr}(\alpha)$ is not a root of $\mathfrak{h}$, then
$\mathfrak{g}_{\pm\alpha}\subset\mathfrak{m}$. Let $v$ be any nonzero vector in
$\mathfrak{g}_{\pm\alpha}$. Then the linear span of $v$ and   the orthogonal complement $\alpha^\perp$ of $\alpha$ in
$\mathfrak{t}$ is a FSCS, which is a contradiction to Theorem \ref{thm-4-2}.

(3)\quad Notice that if the one dimensional space $\mathfrak{t}\cap\mathfrak{h}^{\perp}$ is
contained in $\mathbb{R}\alpha+\mathbb{R}\beta$, then
$\alpha^\perp\cap\beta^\perp$ is a subspace of $\mathfrak{t}\cap\mathfrak{h}$ of codimension one. Now the assertion can be proved similarly as  (2) of Lemma \ref{trick-lemma-1}.
\end{proof}


\section{Positively curved normal homogeneous spaces $M=G/H$ with $\mathrm{rk}G=\mathrm{rk}H$}
\label{s6}

In this section, we assume that $(M,F)$ is a positively curved connected normal homogeneous space induced by the projection $\rho:G\rightarrow M=G/H$ with $\mathrm{rk}G=\mathrm{rk}H$ and a bi-invariant Finsler metric $\bar{F}$ on $G$. We will also assume that $G$ is connected  simply connected, and $H$ is connected.

We first make a simple observation. Suppose $\mathfrak{g}=\mathfrak{g}_1\oplus\mathfrak{g}_2\cdots \oplus\mathfrak{g}_m\oplus\mathbb{R}^l$ is the decomposition of $\mathfrak{g}$ into the direct sum of simple ideals and the Euclidean factor. Then $\mathfrak{h}$ has a similar decomposition which only differs from the above one in exact one  simple factor,  for example, $\mathfrak{h}=\mathfrak{h}_1\oplus\mathfrak{g}_2\oplus\cdots \oplus\mathfrak{g}_m\oplus\mathbb{R}^l$, where $\mathfrak{h}_1$ is a subalgebra of $\mathfrak{g}_1$ with a strictly lower dimension. To prove this statement, we assume conversely that $\mathfrak{h}=\mathfrak{h}_1\oplus\mathfrak{h}_2
\oplus\mathfrak{h}_3\oplus\cdots\oplus\mathfrak{h}_m\oplus\mathbb{R}^l$ where $\dim\mathfrak{h}_i<\dim\mathfrak{g}_i$ for $i=1,2$. Then there exists a nonzero vector $v_i$
from
a root plane $\mathfrak{g}_{\pm\alpha_i}$ in $\mathfrak{m}\cap\mathfrak{g}_i$,
 for $i=1,2$. Thus $v_1$, $v_2$, and the orthogonal
complement $\alpha_1^\perp\cap\alpha_2^\perp$ in $\mathfrak{t}$ span a FSCS, which is a contradiction. This means that  we have the direct sum decompositions,
\begin{eqnarray}
\mathfrak{g}&=&\mathfrak{g}_1\oplus\mathfrak{g}_2\oplus\cdots\oplus\mathfrak{g}_m\oplus
\mathbb{R}^l,\\
\mathfrak{h}&=&\mathfrak{h}_1\oplus\mathfrak{g}_2\oplus\cdots\oplus\mathfrak{g}_m\oplus
\mathbb{R}^l,
\end{eqnarray}
where $\mathfrak{h}_1$ is a subalgebra of the simple compact $\mathfrak{g}_1$ with
$\mathrm{rk}\mathfrak{h}_1=\mathrm{rk}\mathfrak{g}_1$. Let $G_1$ and $H_1$ be the closed connected subgroups with Lie algebras $\mathfrak{g}_1$ and $\mathfrak{h}_1$ respectively.
 Let $\bar{F}_1$
be the restriction of $\bar{F}$ to $G_1$. Then $(M,F)$ is equivalent to the normal homogeneous
space $(M,F_1)$ induced by the projection $\rho:G_1\rightarrow M=G_1/H_1$ and $\bar{F}_1$.
So, to classify the positively curved normal homogeneous spaces in this case, up to
equivalence, we can further assume that $G$ is a compact simple Lie group, i.e., $\mathfrak{g}$ is one of the $A_n$, $B_n$ with $n>1$, $C_n$ with $n>2$, $D_n$ with $n>3$, $E_6$, $E_7$, $E_8$, $F_4$ or $G_2$.

Now we start a case by case study. In the following, we will  presuppose that the   coset space $G/H$  be endowed with an invariant normal Finsler metric of positive flag  curvature. If a contradiction arises, then we can conclude
the corresponding coset space does not admit any positively curved normal homogeneous Finsler metric.
At the final part of the discussion of each case, we will summarize the conclusion.

\medskip
\noindent\textbf{ The case $\mathfrak{g}=A_n$}.\quad The root system $\Delta_\mathfrak{g}$ can be identified with the subset
\begin{equation}\label{root-system-A}
\{\pm(e_i-e_j)| \, 1\leq i<j\leq n+1\}
\end{equation}
in $\mathbb{R}^{n+1}$ with the standard
orthonormal basis
$\{e_1,\ldots,e_{n+1}\}$. The subalgebra $\mathfrak{h}$ is then a direct sum of
the simple compact subalgebras of type $A$ or $\mathbb{R}$ (with abelian Lie bracket). More precisely, there is a decomposition of the
set $\{1,\ldots,n+1\}$ into a non-overlapping union $\mathop{\cup}\limits_{a\in\mathcal{A}}S_a$, such
that for different $a$ and $b$ in $\mathcal{A}$, $\pm (e_i- e_j)$ are not roots of
$\mathfrak{h}$ when $i\in S_a$ and $j\in S_b$, and
for any $a$ in $\mathcal{A}$, $\pm(e_i-e_j)$ are roots of $\mathfrak{h}$ when
$i$, $j\in S_a$ and $i\neq j$.

 Upon the conjugation of a Weyl group action, we can assume that each $S_a$ is a segment of continuous integers. When $S_a$ contains more than one element, it corresponds to a subalgebra of type $A$ in the
direct sum decomposition of $\mathfrak{g}$, otherwise it corresponds to a subalgebra isomorphic
to $\mathbb{R}$.

If $n>2$ and $\mathfrak{h}$ is not of the type $A_{n-1}\oplus\mathbb{R}$,
then upon the conjugation of the  Weyl group action,
$\alpha=e_1-e_{n-1}$ and $\beta=e_2-e_{n}$ are roots of $\mathfrak{m}$, and
$\alpha\pm\beta$ are not roots of $\mathfrak{g}$. However, This is a contradiction to  (1) of Lemma \ref{trick-lemma-1}.

If $n=2$, i.e., $\mathfrak{g}=\mathfrak{su}(3)$, and $\mathfrak{h}$ is not of the type  $A_1\oplus\mathbb{R}$, then upon  the conjugation of the   Weyl group action,
 $\mathfrak{h}$ is a Cartan subalgebra $\mathbb{R}\oplus\mathbb{R}$, and all the roots belong to $\mathfrak{m}$. This is a contradiction to (2) of Lemma \ref{trick-lemma-1}.

In  summarizing, in this case the only possibility is that $\mathfrak{h}=A_{n-1}\oplus\mathbb{R}$ and  $(M,F)$
is equivalent to the normal homogeneous complex projective space
$\mathrm{SU}(n+1)/\mathrm{S}(\mathrm{U}(n)\times \mathrm{U}(1))$. Note that this  is a symmetric coset space of rank one which does admit positively curved normal Finsler (in fact Riemannian) metrics.

\medskip
\noindent\textbf{The case $\mathfrak{g}=B_n$, $n>1$}.\quad The root system can be identified with
the subset
\begin{equation}\label{root-system-B}
\{\pm e_i|\, 1\leq i\leq n\}\cup\{\pm e_i\pm e_j|\,1\leq i<j\leq n\}
\end{equation}
in $\mathbb{R}^n$ with the standard orthonormal basis $\{e_1,\ldots,e_n\}$.
The subalgebra $\mathfrak{h}$ is a direct sum of subalgebras of type $A$, $B$, $D$ or $\mathbb{R}$. To be precise,
there is a decomposition of the
set $\{1,\ldots,n\}$ into a non-overlapping union $\mathop{\cup}\limits_{a\in\mathcal{A}}S_a$, such
that for different $a$ and $b$ in $\mathcal{A}$, $\pm e_i\pm e_j$ are not roots of
$\mathfrak{h}$ when $i\in S_a$ and $j\in S_b$, and
for any $a$ in $\mathcal{A}$, one of the following holds:
\begin{description}
\item{\rm (1)} For any $ i\in S_a$, $\pm e_i$ are roots of $\mathfrak{h}$, and for any $ i,j\in S_a$ with $i\neq j$, $\pm e_i \pm e_j$ are roots of $\mathfrak{h}$;
\item{\rm (2)} For any $i,j\in S_a$ with $i\neq j$, $\pm e_i \pm e_j$ are roots of $\mathfrak{h}$,
but $\pm e_i$s are not roots of $\mathfrak{h}$;
\item{\rm (3)} For any $ i,j\in S_a$ with $i\neq j$, either $\pm (e_i+e_j)$ or $\pm(e_i-e_j)$, but not both, are
roots of $\mathfrak{h}$,  and $\pm e_i$s are not roots of $\mathfrak{h}$;
\item{\rm (4)} $S_a=\{i\}$ and $\pm e_i$ are roots of $\mathfrak{g}$;
\item{\rm (5)} $S_a=\{i\}$ and there is no root for  $S_a$.
\end{description}
 Upon the Weyl group action, we can assume that each $S_a$ is a segment of continuous integers, and for $S_a$ in (3), $\pm(e_i-e_j)$ are roots of $\mathfrak{h}$ but $\pm(e_i+e_j)$ are not. Then each
$S_a$ corresponds to a subalgebra of type $B$, $D$, $A\oplus\mathbb{R}$ (with a one dimensional center), $A_1$ or $\mathbb{R}$.

If $n>3$ and $\mathfrak{h}$ is not of the type $B_{n-1}\oplus\mathbb{R}$, $D_{n-1}\oplus\mathbb{R}$ or $D_n$, then we can take
$\alpha=e_1+e_{n-1}$ and $\beta=e_2+e_n$ from roots of $\mathfrak{m}$. Then $\alpha\pm\beta$ are not roots of $\mathfrak{g}$, which is a contradiction to
 (1) of Lemma \ref{trick-lemma-1}.
If $\mathfrak{h}$ is $B_{n-1}\oplus\mathbb{R}$ or $D_{n-1}\oplus\mathbb{R}$ (here,  without losing generality,
we assume that $\mathbb{R}$ corresponds to $\{n\}$, and we only need $n>2$), then we can take
$\alpha=e_1+e_n$ and $\beta=e_1-e_n$. This is also a contradiction to  (1) of Lemma \ref{trick-lemma-1}.

If $n=3$ and $\mathfrak{h}$ has an $A_1$ or a $\mathbb{R}$ factor from (4) or (5) of the list
in it, then we can also use
$\alpha=e_1+e_n$ and $\beta=e_1-e_n$ to deduce a contradiction.
If $\mathfrak{h}=A_2\oplus\mathbb{R}$ from (3) of the list, then we can take $\alpha=e_1+e_2$ and $\beta=e_3$
to deduce a contradiction to  (1) of Lemma \ref{trick-lemma-1}. Thus $\mathfrak{h}$ can
only be $D_3$.

If $n=2$ and $\mathfrak{h}$ have a $A_1$ or $\mathbb{R}$ factor from (4) or (5) of the list, then we can take $\alpha=e_1+e_2$ and $\beta=e_1-e_2$
to deduce a  contradiction  to  (1) of Lemma \ref{trick-lemma-1}.
Now besides $D_2$,
$\mathfrak{h}$ can also be $A_1\oplus \mathbb{R}$ from (3) of the list, where
$(M,F)$ is equivalent to a homogeneous complex projective space
$\mathbb{C}\mathrm{P}^{3}=\mathrm{Sp}(2)/\mathrm{Sp}(1)S^1$.

In summarizing,  in this case  there are only two possibilities, namely,  $(M,F)$ is
equivalent to a normal homogeneous sphere $S^{2n}=\mathrm{SO}(2n+1)/\mathrm{SO}(2n)$,  or a homogeneous complex projective space
$\mathbb{C}\mathrm{P}^{3}=\mathrm{Sp}(2)/\mathrm{Sp}(1)S^1$. Note that the first one  is a symmetric coset space of rank one, and the second one appears in Berger's list, hence they both admit positively curved normal Finsler (in fact Riemannian) metrics.

\medskip
\noindent \textbf{The case $\mathfrak{g}=C_n$, $n>2$}.\quad The root system of $\mathfrak{g}$ can be
identified with the subset
\begin{equation}\label{root-system-C}
\{\pm 2 e_i|\, 1\leq i\leq n\}\cup\{\pm e_i\pm e_j|\, 1\leq i<j\leq n\}
\end{equation}
in $\mathbb{R}^n$ with the standard orthonormal basis $\{e_1,\ldots,e_n\}$.
The subalgebra $\mathfrak{h}$ is a direct sum of subalgebras of type $C$, $A$ and $\mathbb{R}$. To be precise,
there is a decomposition of the
set $\{1,\ldots,n\}$ into a non-overlapping union $\mathop{\cup}\limits_{a\in\mathcal{A}}S_a$, such
that for different $a$ and $b$ in $\mathcal{A}$, $\pm e_i\pm e_j$ are not roots of
$\mathfrak{h}$ when $i\in S_a$ and $j\in S_b$, and
for any $a$ in $\mathcal{A}$, one of the following holds:
\begin{description}
\item{\rm (1)} For any $ i\in S_a$, $\pm 2 e_i$ are roots of $\mathfrak{h}$, and for any $i,j\in S_a$ with $i\neq j$, $\pm e_i \pm e_j$ are roots of $\mathfrak{h}$;
\item{\rm (2)} For any $ i,j\in S_a$ with $i\neq j$, either $\pm(e_i+e_j)$ or $\pm (e_i-e_j)$, but not both, are roots
of $\mathfrak{h}$, and $\pm 2 e_i$s are not roots of $\mathfrak{h}$;
\item{\rm (3)} $S_a=\{i\}$ and $\pm 2 e_i$ are roots of $\mathfrak{h}$;
\item{\rm (4)} $S_a=\{i\}$ and there are no roots for  $S_a$.
\end{description}
Upon the  Weyl group action, we can assume that each $S_a$ is a segment of continuous integers, and for
any $S_a$ in (2), $\pm(e_i-e_j)$ are roots of $\mathfrak{h}$ but
$\pm(e_i+e_j)$ are not. Then each $S_a$ corresponds to a subalgebra of type $C$,
$A\oplus\mathbb{R}$ (with a one dimensional center), $A_1$ or $\mathbb{R}$.

If $n>3$ and $\mathfrak{h}$ is not $C_{n-1}\oplus \mathbb{R}$ or $C_{n-1}\oplus A_1$, then considering $\alpha=e_1+e_{n-1}$ and $\beta=e_2+e_n$ we get
a contradiction to (1) of Lemma \ref{trick-lemma-1}.

 Suppose $n=3$ and $\mathfrak{h}=A_1\oplus A_1\oplus \mathbb{R}$, where the $\mathbb{R}$ factor is from (4). Then  we can assume that $\mathbb{R}$ corresponds to $\{3\}$. Considering $\alpha=e_1 + e_2$ and $\beta=2 e_3$, we get a
contradiction to (1) of Lemma \ref{trick-lemma-1}. If $\mathfrak{h}=A_1\oplus (A_1\oplus\mathbb{R})$, where the factor $A_1\oplus\mathbb{R}$ is from (2) of the list and assumed to correspond to $\{2,3\}$, then we can still use $\alpha=e_1+e_2$ and $\beta=2e_3$ to deduce a contradiction. Moreover, if $n=3$ and
$\mathfrak{h}=A_1\oplus A_1\oplus A_1$, then considering the roots $\pm(e_1-e_2)$, $\pm (e_2-e_3)$ and
$\pm (e_1-e_3)$ in $\mathfrak{m}$, we get a   contradiction to (2) of Lemma \ref{trick-lemma-1}.

In summarizing, in this case the only possibility is that $\mathfrak{h}=C_{n-1}\oplus \mathbb{R}$ or $C_{n-1}\oplus A_1$.
 Thus  $(M,F)$ is equivalent to the normal homogeneous complex projective space $\mathrm{Sp}(n)/\mathrm{Sp}(n-1)S^1$ or the normal homogeneous quaternionic projective space
$\mathrm{Sp}(n)/\mathrm{Sp}(n-1)\mathrm{Sp}(1)$. It is well known that both the above spaces
admit positively curved normal Finsler (in fact Riemannian) metrics.

\medskip
\noindent \textbf{ The case $\mathfrak{g}=D_n$, $n>3$}.\quad
The root system of $\mathfrak{g}$ can be identified with the subset
\begin{equation}\label{root-system-D}
\{\pm e_i\pm e_j|\, 1\leq i<j\leq n\}
\end{equation}
in $\mathbb{R}^n$ with the standard
orthonormal basis $\{e_1,\ldots,e_n\}$. The subalgebra $\mathfrak{h}$ is a direct sum
of subalgebras of type $D$, $A$ and $\mathbb{R}$. To be precise,
there is a decomposition of the
set $\{1,\ldots,n\}$ into a non-overlapping union $\mathop{\cup}\limits_{a\in\mathcal{A}}S_a$, such
that for different $a$ and $b$ in $\mathcal{A}$, $\pm e_i\pm e_j$ are not roots of
$\mathfrak{h}$ when $i\in S_a$ and $j\in S_b$, and
for any $a$ in $\mathcal{A}$, one of the following holds:
\begin{description}
\item{\rm (1)} For any $i,j\in S_a$ with $i\neq j$, $\pm e_i\pm e_j$ are roots of $\mathfrak{h}$;
\item{\rm (2)} For any $i,j\in S_a$ with $i\neq j$, either $\pm(e_i+e_j)$ or $\pm(e_i-e_j)$,  but not both, are
roots of $\mathfrak{h}$;
\item{\rm (3)} $S_a=\{i\}$ and there is no root for $S_a$.
\end{description}
 Upon the automorphisms of $D_n$ induced by the Weyl group action of $B_n$, we can assume that each $S_a$ is a segment of continuous integers, and for
any $S_a$ of the second case, $\pm(e_i-e_j)$ are roots of $\mathfrak{h}$ but
$\pm(e_i+e_j)$ are not. Then each $S_a$ corresponds to a subalgebra of type $D$,
$A\oplus\mathbb{R}$ (with a one dimensional center) or $\mathbb{R}$.

If $\mathfrak{h}$ is not $D_{n-1}\oplus \mathbb{R}$,  then upon the conjugation of the Weyl group action, we can select $\alpha=e_1+e_{n-1}$ and $\beta=e_2+e_n$ and deduce a contradiction to (1) of
Lemma \ref{trick-lemma-1}. If $\mathfrak{h}=D_{n-1}\oplus\mathbb{R}$ (here we can assume that
$\mathbb{R}$
corresponds to $\{n\}$),
then we can select $\alpha=e_1+e_n$ and $\beta=e_1-e_n$ and deduce a contradiction to (1)
of Lemma \ref{trick-lemma-1}.

In conclusion, in this case, there does not exist  any positively curved normal homogeneous Finsler space.

\medskip
\noindent\textbf{The case $\mathfrak{g}=E_6$, $E_7$ or $E_8$}.\quad First note that for any two different
root planes $\mathfrak{g}_{\pm\alpha}$
and $\mathfrak{g}_{\pm\beta}$ in $\mathfrak{m}$, the angle between $\alpha$ and $\beta$
can not be $\pi/2$, otherwise $\alpha\pm\beta$ are not roots of $\mathfrak{g}$ and
we can deduce a contradiction to (1) of Lemma \ref{trick-lemma-1}. If the angle is
$\pi/3$, then $\pm(\alpha+\beta)$ are not roots of $\mathfrak{g}$ and
$[\mathfrak{g}_{\pm\alpha},\mathfrak{g}_{\pm\beta}]=\mathfrak{g}_{\pm(\alpha-\beta)}$ is
a root plane.
Then $ \mathfrak{g}_{\pm(\alpha-\beta)}$  must be a root plane in $\mathfrak{h}$, otherwise we can get a contradiction to
(2) of Lemma \ref{trick-lemma-1}. Moreover, for any root $\alpha$ of $\mathfrak{m}$, we have $$[\mathfrak{g}_{\pm\alpha},
\mathfrak{g}_{\pm\alpha}]\subset\mathfrak{t}\subset\mathfrak{h}.$$
   Consequently we have $[\mathfrak{m},\mathfrak{m}]\subset\mathfrak{h}$. By Proposition \ref{prop-5-2}, $(M,F)$ is a symmetric normal homogeneous space of rank $1$, which is impossible (see \cite{HE78}). In fact, this
argument is also valid for the cases of $A_n$ and $D_n$. In conclusion, in this case there does not exist any positively curved normal homogeneous Finsler spaces.

\medskip
\noindent\textbf{The case $\mathfrak{g}=F_4$}.\quad
The root system can be identified with the subset
\begin{equation}\label{root-system-F}
\{\pm e_i|\, 1\leq i\leq 4\}\cup\{\pm e_i\pm e_j|\, 1\leq i<j\leq 4\}\cup\{
\frac12(\pm e_1\pm e_2\pm e_3\pm e_4\}
\end{equation}
in $\mathbb{R}^4$ with the standard orthonormal basis $\{e_1,e_2,e_3,e_4\}$.
If there are
two short roots $\alpha$ and $\beta$ of $\mathfrak{m}$ with $\alpha\neq\beta$,
then using the  Weyl group action (which changes $\mathfrak{h}$ by a conjugation, without
changing the equivalent class of the normal homogeneous space $M$) we can assume that
$\alpha=\frac{1}{2}(e_1+e_2+e_3+e_4)$. By (1) of Lemma \ref{trick-lemma-1}, all long roots of the form
$\pm (e_i-e_j)$ belongs to $\mathfrak{h}$. Taking into account the  root $\beta$, we conclude that there are more long roots
belonging to $\mathfrak{h}$ by (1) of Lemma \ref{trick-lemma-1}. Thus
all long roots belong to $\mathfrak{h}$.
There are only two choices left for $\mathfrak{h}$, namely, $\mathfrak{h}=B_4$ if it has a short
root, or $D_4$ if it has no short root. If $\mathfrak{h}=D_4$, then the plane
spanned by $e_1$ and $\frac{1}{2}(e_1+e_2+e_3+e_4)$ contains exactly 6 roots which all belong to $\mathfrak{m}$, i.e. $\pm e_1$, $\pm\frac{1}{2}(e_1+e_2+e_3+e_4)$ and
$\pm\frac{1}{2}(-e_1+e_2+e_3+e_4)$ which gives a root system of $A_2$.
But then we can deduce a
contradiction to (2) of Lemma \ref{trick-lemma-1}, by setting $\alpha=e_1$ and $\beta=\frac{1}{2}(e_1+e_2+e_3+e_4)$. If $\mathfrak{h}=B_4$, then by
a suitable conjugation of the Weyl group action of $\mathfrak{g}$, we can assume that all
roots $\pm e_i$ belong to $\mathfrak{h}$, and all other roots $\frac12(\pm e_1\pm e_2\pm e_3\pm e_4)$ belong to $\mathfrak{m}$. Then it is easy to see that $(M,F)$ is equivalent to the symmetric normal homogeneous Cayley plane $F_4/\mathrm{Spin}(9)$. If $\mathfrak{m}$ contains at most one pair of short roots $\pm\alpha$ of $\mathfrak{g}$,
then it is easily seen that $\mathfrak{h}=\mathfrak{g}$, which is a contradiction.

In conclusion, in this case the only possibility is that $(M,F)$ is equivalent to the symmetric normal homogeneous Cayley plane $F_4/\mathrm{Spin}(9)$. Since it is a symmetric coset space of rank one, there does exist positively curved normal Finsler (Riemannian) metric on it.

\medskip
\noindent\textbf{The case $\mathfrak{g}=G_2$}. In this case it is  easily  seen that $\mathfrak{h}=A_2$, $A_1\oplus A_1$, $A_1\oplus \mathbb{R}$, or $\mathbb{R}\oplus\mathbb{R}$.
By (1) of Lemma \ref{trick-lemma-1}, $\mathfrak{m}$ can not have a pair of roots
$\alpha$ and $\beta$ with  angle $\pi/2$. Then $\mathfrak{h}$ must be $A_2$. Hence
the only possibility  is that $(M,F)$ is equivalent to the normal homogeneous
sphere $S^6=G_2/\mathrm{SU}(3)$. This coset space appears in Berger's list and it admits
positively curved normal Finsler (Riemannian) metrics.

The above discussion can be summarized as the following proposition.

\begin{proposition} \label{summary-prop}
Let $(M,F)$ be a positively curved normal homogeneous space, with $M=G/H$ and
$\mathrm{rk}G=\mathrm{rk}H$.  Then we have
\begin{description}
\item{\rm (1)} In the sense of equivalent coset spaces, $(M,F)$ must be equivalent to one of the coset spaces
$\mathbb{C}\mathrm{P}^{n}=\mathrm{SU}(n+1)/S(\mathrm{U}(n)\times\mathrm{U}(1))$ with $n\geq 1$,
$S^{2n}=\mathrm{SO}(2n+1)/\mathrm{SO}(2n)$ with $n>1$,
$\mathbb{H}\mathrm{P}^n=\mathrm{Sp}(n+1)/\mathrm{Sp}(n)\mathrm{Sp}(1)$ with $n>1$, $\mathbb{C}\mathrm{P}^{2n+1}=\mathrm{Sp}(n+1)/\mathrm{Sp}(n)S^1$ with $n\geq 1$,
or $S^6=\mathrm{G}_2/\mathrm{SU}(3)$.
\item{\rm (2)} In the Lie algebra level, under the assumption that $\mathfrak{g}$ is  compact simple,  the pair
$(\mathfrak{g},\mathfrak{h})$ must be one of  $(A_n,A_{n-1}\oplus\mathbb{R})$ with $n\geq 1$,
$(B_n,D_n)$ with $n>1$ (notice that $B_2=C_2$ and $D_2=A_1\oplus A_1$), $(C_n,C_{n-1}\oplus A_1)$ with $n>2$, $(C_n,C_{n-1}\oplus\mathbb{R})$ with $n>1$ (notice that $(C_2,A_1\oplus\mathbb{R})=
(B_2,A_1\oplus\mathbb{R})$), or $(G_2,A_2)$.
\end{description}
\end{proposition}
\begin{remark}
In \cite{XDHH2014}, we (with two other co-authors) have classified all positively
curved homogeneous Finsler spaces satisfying the condition $\mathrm{rk}G=\mathrm{rk}H$, and
get exactly the same list as in \cite{Wallach1972}. Without the normal condition, we only have
(1) of Lemma \ref{trick-lemma-1}. The homogeneous spaces $\mathrm{SU}(3)/\mathrm{T}^2$,
$\mathrm{Sp}(3)/\mathrm{Sp(1)}\times \mathrm{Sp}(1)\times \mathrm{Sp}(1)$ and $F_4/\mathrm{Spin}(8)$, which can be endowed with positively curved
homogeneous non-Riemannian metrics,  can not be endowed with any positively curved normal
homogeneous metrics by  (2) of Lemma \ref{trick-lemma-1}.
\end{remark}
\section{Positively curved normal homogeneous spaces $M=G/H$ with $\mathrm{rk}G=\mathrm{rk}H+1$}
\label{s7}

In the above section, we have determined positively curved normal homogeneous Finsler spaces $G/H$
where $G$ and $H$ have the same rank. In this section, we consider the  other case. We first give the theme for the study.
\subsection{The theme for the study}
In this section, we assume that $(M,F)$ is a connected positively curved normal homogeneous space induced by the projection $\rho:G\rightarrow M=G/H$ with $\mathrm{rk}G=\mathrm{rk}H+1$ and a bi-invariant Finsler metric $\bar{F}$ on $G$. We will also assume that $G$ is connected and simply connected, and
$H$ is connected. We keep the notations of the above sections.

In the following, we will frequently consider the subalgebra $\mathfrak{k}$ of $\mathfrak{g}$ generated by $\mathfrak{h}$ and $\mathfrak{t}$. Since $\mathfrak{k}$  contains the
Cartan subalgebra $\mathfrak{t}$, there is a connected closed subgroup $K$ of $G$
such that $\mathrm{Lie}(K)=\mathfrak{k}$.
 If $\dim\mathfrak{k}<\dim \mathfrak{g}$, then the discussion in Theorem \ref{thm-4-2} shows that the corresponding coset space $G/K$ admits positively curved normal homogeneous Finsler metrics. Thus $G/K$ must be one of the coset spaces in the list of Proposition \ref{summary-prop}.

It is obvious that we only need to consider the following three cases:
\begin{description}
\item{I}\quad Each root plane of $\mathfrak{h}$ is a root plane of $\mathfrak{g}$;
\item{II}\quad There is a root plane $\mathfrak{h}_{\pm\alpha'}$ of $\mathfrak{h}$ which is
not a root plane of $\mathfrak{g}$, and there are two roots $\alpha$ and $\beta$ of $\mathfrak{g}$ belonging to
different simple components of $\mathfrak{g}$ such that $\mathrm{pr}(\alpha)=\mathrm{pr}(\beta)=\alpha'$.
\item{III}\quad There is a root plane $\mathfrak{h}_{\pm\alpha'}$ of $\mathfrak{h}$ which is
not a root plane of $\mathfrak{g}$, and there are two roots $\alpha$ and $\beta$ of
$\mathfrak{g}$ belonging to the same simple component such that
$\mathrm{pr}(\alpha)=\mathrm{pr}(\beta)=\alpha'$.
\end{description}

Now we start a case by case study. As in Section \ref{s6}, if a contradiction
arises, then we conclude that the coset space under consideration does not admit any
positively curved normal Finsler metric.

\noindent \textbf{Case I}.\,\,\,\quad We assume that each root plane of $\mathfrak{h}$ is a root plane of $\mathfrak{g}$. Then
the subalgebra $\mathfrak{k}$ has
the same root planes as $\mathfrak{h}$. Moreover,  $\mathfrak{h}$ has the same semi-simple
components as $\mathfrak{k}$. If $\mathfrak{k}=\mathfrak{g}$, then $\dim M=1$,
which is a contradiction. If $\mathfrak{k}\neq \mathfrak{g}$, then by Theorem \ref{thm-4-2}, we can apply the discussion in Section \ref{s6} to the positively curved normal homogeneous space $G/K$. So we have the following direct sum
decompositions
\begin{eqnarray}
\mathfrak{g}&=&\mathfrak{g}_1\oplus\mathfrak{g}_2\oplus\cdots\oplus\mathfrak{g}_m\oplus
\mathbb{R}^l,\\
\mathfrak{k}&=&\mathfrak{k}_1\oplus\mathfrak{g}_2\oplus\cdots\oplus\mathfrak{g}_m\oplus
\mathbb{R}^l,
\end{eqnarray}
where the pair $(\mathfrak{g}_1,\mathfrak{k}_1)$ can only be
$(A_n,A_{n-1}\oplus \mathbb{R})$, $(B_n,D_n)$, $(C_n,C_{n-1}\oplus A_1)$,
$(C_n,C_{n-1}\oplus \mathbb{R})$, $(F_4,B_4)$ or $(G_2,A_2)$ (see Proposition \ref{summary-prop}). Note that the cases
that $(\mathfrak{g}_1,\mathfrak{k}_1)=(B_n,D_n)$, $(C_n,C_{n-1}\oplus A_1)$,
$(F_4,B_4)$ and $(G_2,A_2)$ are impossible. Otherwise we will have
\begin{equation}
\mathfrak{h}=\mathfrak{k}_1\oplus\mathfrak{g}_2\oplus\cdots\oplus\mathfrak{g}_m\oplus
\mathbb{R}^{l-1}.
\end{equation}
Choose an arbitrary nonzero vector $v$ in a root plane $\mathfrak{g}_{\pm\alpha}$
in $\mathfrak{m}$. Then $v$ and $\alpha^\perp$ in $\mathfrak{t}$
span a FSCS, which is a contradiction to Theorem \ref{thm-4-2}. For the remaining
two cases $(\mathfrak{g}_1,\mathfrak{k}_1)=(A_n,A_{n-1}\oplus\mathbb{R})$ and
$(\mathfrak{g}_1,\mathfrak{k}_1)=(C_n,C_{n-1}\oplus\mathbb{R})$, the corresponding $M=G/H$ is a $S^1$-bundle over the complex projective space $G/K$, which is equivalent to
its universal covering, i.e.,
a normal homogeneous sphere $\mathrm{SU}(n+1)/\mathrm{SU}(n)$, $\mathrm{U}(n+1)/\mathrm{U}(n)$,
$\mathrm{Sp}(n)/\mathrm{Sp}(n-1)$, or $\mathrm{Sp}(n)S^1/\mathrm{Sp}(n-1)S^1$.
These spaces all appear in Berger's list, so they all admit positively curved normal Finsler (Riemannian) metrics.

\medskip
\noindent\textbf{Case II}.\quad We assume that
there is a root plane $\mathfrak{h}_{\pm\alpha'}$  which is
not a root plane of $\mathfrak{g}$, and there are two roots $\alpha$ and $\beta$ of $\mathfrak{g}$ belonging to
different simple components of $\mathfrak{g}$, say $\mathfrak{g}_1$ and $\mathfrak{g}_2$,
 such that $\mathrm{pr}(\alpha)=\mathrm{pr}(\beta)=\alpha'$.
Then $\hat{\mathfrak{g}}_{\pm\alpha'}=(\mathfrak{g}_1)_{\pm\alpha}+(\mathfrak{g}_2)_{\pm\beta}$,
and
all other root planes of $\mathfrak{h}$ are also root planes of $\mathfrak{g}$.
We claim that
$\mathfrak{h}_{\pm\alpha'}\cap(\mathfrak{g}_1)_{\pm\alpha}=\mathfrak{h}_{\pm\alpha'}
\cap(\mathfrak{g}_2)_{\pm\beta}=0$. In fact,  otherwise the $\mathrm{Ad}(\exp (\mathfrak{t}\cap\mathfrak{h}))$ actions will make $\mathfrak{h}_{\pm\alpha'}$
a root plane of $\mathfrak{g}$.

Now we have the direct sum decompositions
\begin{eqnarray}
\mathfrak{g}&=&\mathfrak{g}_1\oplus\mathfrak{g}_2\oplus\mathfrak{g}_3\oplus\cdots
\oplus\mathfrak{g}_m\oplus\mathbb{R}^l,\\
\mathfrak{h}&=&\mathfrak{h}'\oplus\mathfrak{g}_3\oplus\cdots
\oplus\mathfrak{g}_m\oplus\mathbb{R}^l,
\end{eqnarray}
where $\mathfrak{h}'$ is a subalgebra of $\mathfrak{g}_1\oplus\mathfrak{g}_2$. We now prove that, if there
is any other root $\alpha''$ of $\mathfrak{g}_1$, such that $(\mathfrak{g}_1)_{\pm\alpha''}$
is a root plane of $\mathfrak{h}'$,
then $[(\mathfrak{g}_1)_{\pm\alpha''},(\mathfrak{g}_1)_{\pm\alpha}]=0$, from
which we can also see $\alpha$ and $\alpha''$ must be orthogonal to each other. In fact, otherwise we have
$$[(\mathfrak{g}_1)_{\pm\alpha''},\mathfrak{h}_{\pm\alpha'}]=
[(\mathfrak{g}_1)_{\pm\alpha''},(\mathfrak{g}_1)_{\pm\alpha}]\neq 0.$$ Then
$$(\mathfrak{g}_1)_{\pm\alpha}\subset
[(\mathfrak{g}_1)_{\pm\alpha''},[(\mathfrak{g}_1)_{\pm\alpha''},\mathfrak{h}_{\pm\alpha'}]]\subset
\mathfrak{h}',$$
which is a contradiction. Similarly, if there is any other root $\beta''$ of $\mathfrak{g}_2$ such that $(\mathfrak{g}_2)_{\pm\beta''}$ is a root plane of $\mathfrak{h}'$, then $[(\mathfrak{g}_2)_{\pm\beta''},(\mathfrak{g}_2)_{\pm\beta}]=0$, and $\beta''$
is orthogonal to $\beta$.

The above argument shows that, if there exists any root of $\mathfrak{h}'$ other than $\pm\alpha'$,
then the subalgebra $\mathfrak{k}$ must be the   direct sum of three subalgebras in
$\mathfrak{g}_1\oplus\mathfrak{g}_2$. This implies that $\mathfrak{k}\neq \mathfrak{g}$.
According to the above arguments, we have direct sum decompositions:
\begin{eqnarray}
\mathfrak{g}&=&(\mathfrak{g}_1\oplus\mathfrak{g}_2)\oplus\mathfrak{g}_3\oplus\cdots
\oplus\mathfrak{g}_m\oplus\mathbb{R}^l,\label{0005}\\
\mathfrak{k}&=&(\mathfrak{h}''\oplus A_1\oplus A_1)\oplus\mathfrak{k}_3\oplus\cdots\oplus\mathfrak{k}_n\oplus\mathbb{R}^l,\\
\mathfrak{h}&=&(\mathfrak{h}''\oplus A_1)\oplus\mathfrak{k}_3\oplus\cdots\oplus\mathfrak{k}_n\oplus\mathbb{R}^l,\label{0007}
\end{eqnarray}
where each $\mathfrak{k}_i$ is a subalgebra of $\mathfrak{g}_i$. Moreover, in the decomposition of $\mathfrak{k}$,
$\mathfrak{h}''\oplus A_1\oplus A_1$ is a subalgebra of $\mathfrak{g}_1\oplus\mathfrak{g}_2$, and the two $A_1$-factors correspond to
the roots $\alpha$ and $\beta$, respectively. Furthermore,   in the decomposition of $\mathfrak{h}$, the $A_1$-factor
is a diagonal subalgebra of $A_1\oplus A_1$ in
$\mathfrak{k}$, corresponding to the root $\alpha'$.
Applying the same arguments in Section \ref{s6} to $G/K$, we can assume, up to a  proper reorder, that $\mathfrak{k}_i=\mathfrak{g}_i$ for $i>2$, and $\mathfrak{g}_2=A_1$ or $\mathfrak{h}''$.

By Proposition \ref{summary-prop}, there are only very
limited choices for $\{\mathfrak{g}_1,\mathfrak{g}_2\}$ and $\mathfrak{h}''$. More precisely, if
$\mathfrak{g}_2=A_1$ and $\mathfrak{h}''=\mathbb{R}$, then we have
$\mathfrak{g}_1=A_2$, and $(M,F)$ is equivalent to the normal homogeneous space $\mathrm{SU}(3)\times \mathrm{SO}(3)/\mathrm{U}^*(2)$ constructed in \cite{WILKING1990}.
If $\mathfrak{g}_2=A_1$ and $\mathfrak{h}''=A_1$, then we have $\mathfrak{g}_1=C_2$, and
$(M,F)$ is equivalent to the homogeneous sphere
$S^7=\mathrm{Sp}(2)\mathrm{Sp}(1)/\mathrm{Sp}(1)\mathrm{Sp}(1)$.
Furthermore, if $\mathfrak{g}_2=A_1$ and $\mathfrak{h}''=C_{n-1}$ with $n>2$, then we have
$\mathfrak{g}_1=C_n$, and $(M,F)$ is equivalent to the homogeneous sphere
$S^{4n-1}=\mathrm{Sp}(n)\mathrm{Sp}(1)/\mathrm{Sp}(n-1)\mathrm{Sp}(1)$.
If $\mathfrak{g}_2=\mathfrak{h}''$, then $\mathfrak{g}_1$ must be $C_2$, and
we have $S^7=\mathrm{Sp}(2)\mathrm{Sp}(1)/\mathrm{Sp}(1)\mathrm{Sp}(1)$
as well.


If there does not exist any  root other than $\pm\alpha'$ for $\mathfrak{h}'$, then
(\ref{0005})-(\ref{0007}) are still valid for $\mathfrak{h}''=0$.
If $\mathfrak{k}\neq\mathfrak{g}$, then by the discussion in Section \ref{s6}
and Proposition \ref{summary-prop}, the pair $\{\mathfrak{g}_1,\mathfrak{g}_2\}$ must be
$\{0,C_2\}$, which is a contradiction with our assumption. So $\mathfrak{g}_1=\mathfrak{g}_2=A_1$,
and $(M, F)$ is
equivalent to the symmetric coset space $S^3=\mathrm{SO}(4)/\mathrm{SO}(3)$.

\medskip
\noindent\textbf{Case III}.\quad We assume that there is a root plane $\mathfrak{h}_{\pm\alpha'}$ of $\mathfrak{h}$ which is
not a root plane of $\mathfrak{g}$, and there are two roots $\alpha$ and $\beta$ of
$\mathfrak{g}$ belonging to the same simple component, say $\mathfrak{g}_1$, such that
$\mathrm{pr}(\alpha)=\mathrm{pr}(\beta)=\alpha'$.
From our description for $\mathfrak{h}$ we have the direct sum decompositions
\begin{eqnarray}
\mathfrak{g}=\mathfrak{g}_1\oplus\mathfrak{g}_2\oplus\cdots\oplus\mathfrak{g}_{m}
\oplus\mathbb{R}^l,\\
\mathfrak{h}=\mathfrak{h}_1\oplus\mathfrak{h}_2\oplus\cdots\oplus\mathfrak{h}_{m}
\oplus\mathbb{R}^l,
\end{eqnarray}
where $\mathfrak{h}_i$ is a subalgebra of the simple factor $\mathfrak{g}_i$ with
$\mathrm{rk}\mathfrak{h}_i=\mathrm{rk}\mathfrak{g}_i$, $\forall i$. We assert that  $\mathfrak{h}_i=\mathfrak{g}_i$ for any $i>1$. In fact, otherwise we can find a root plane
$\mathfrak{g}_{\pm\gamma}$ in $\mathfrak{m}$ from some $\mathfrak{g}_i$, $i>1$. Then for any
nonzero vector $v\in\mathfrak{g}_{\pm\gamma}$, the linear span of $v$ and the orthogonal
complement $\gamma^{\perp}$ in $\mathfrak{t}$ is a FSCS. This is a contradiction to
Theorem \ref{thm-4-2}.
Now let $H_1$ be the connected subgroup of $G_1$ with $\mathrm{Lie}(H_1)=\mathfrak{h}_1$. Then
$(M,F)$ is equivalent to a normal homogeneous space $G_1/H_1$. Thus we can assume that $\mathfrak{g}$ is a simple compact Lie algebra.

There are some common cases which can be excluded by the following lemmas.
\begin{lemma}\label{trick-lemma-3} Assume that $G$ is a compact simple Lie group which is not $G_2$, and $M=G/H$
is a positively curved normal homogeneous space in Case III. Keep all the above notations. Then the angle
between $\alpha$ and $\beta$ can not be $\frac{\pi}3$ or $\frac{2\pi}3$.
\end{lemma}
\begin{proof}
The condition that $\mathfrak{g}$ is not $G_2$ and the angle
between $\alpha$ and $\beta$ is $\frac{\pi}3$ or $\frac{2\pi}3$ indicates that
the subalgebra algebra $\mathfrak{g}'=\mathbb{R}\alpha+\mathbb{R}\beta+
\mathop{\sum}\limits_{m,n\in\mathbb{Z}}\mathfrak{g}_{\pm(m\alpha+n\beta)}$ is isomorphic to $A_2$, and
$\mathfrak{h}\cap\mathfrak{g}'=\mathbb{R}\alpha'+\mathfrak{h}_{\pm\alpha'}$ is isomorphic to $A_1$.
If the angle between $\alpha$ and $\beta$ is
$\frac{\pi}3$, then we can naturally present $\mathfrak{h}$ by matrices in $\mathfrak{su}(3)$,  i.e., $\mathfrak{h}\cap\mathfrak{g}'$ is linearly spanned by the following three nonzero matrices:
$$u=\left(
  \begin{array}{ccc}
    \sqrt{-1} & 0 & 0 \\
        0 & \sqrt{-1} & 0 \\
        0 &     0 & -2\sqrt{-1} \\
  \end{array}
\right)\in \mathbb{R}\alpha',\quad
v=\left(
    \begin{array}{ccc}
      0 & 0 & a \\
      0 & 0 & b \\
      -\bar{a} & -\bar{b} & 0 \\
    \end{array}
  \right)\in\mathfrak{h}_{\alpha'},
$$
and
$$w=\frac13 [u,v]=\left(
            \begin{array}{ccc}
              0 & 0 & \sqrt{-1}a \\
              0 & 0 & \sqrt{-1}b \\
              \sqrt{-1}\bar{a} & \sqrt{-1}\bar{b} & 0 \\
            \end{array}
          \right)\in\mathfrak{h}_{\alpha'}.
$$
But then the matrix
$$\frac12 [v,w] =
\left(
  \begin{array}{ccc}
    \sqrt{-1}|a|^2 & \sqrt{-1}a\bar{b} & 0 \\
    \sqrt{-1}\bar{a}b & \sqrt{-1}|b|^2 & 0 \\
    0 & 0 & -\sqrt{-1}|a|^2-\sqrt{-1}|b|^2 \\
  \end{array}
\right)
$$
can belong to $\mathfrak{h}\cap\mathfrak{g}'$ only when $a=b=0$,
which is a contradiction.

The proof for the case that the angle between $\alpha$ and $\beta$ is $\frac{2\pi}3$  is similar.
\end{proof}

Now we continue the  case by case study for Case III.
\subsection{The exceptional cases}
\noindent\textbf{The case $\mathfrak{g}=G_2$}.\quad The root system of $G_2$ can be identified with the subset
\begin{equation}
\{(\pm\sqrt{3},0),(\pm\frac{\sqrt{3}}{2},\pm\frac{3}{2}),
(0,\pm 1),(\pm \frac{\sqrt{3}}{2},\pm\frac{1}{2})\}
\end{equation}
in $\mathbb{R}^2$ with the standard inner product. Since $\mathfrak{h}=A_1$, all roots
of $\mathfrak{g}$ outside the subset
$\mathrm{pr}^{-1}(\pm\alpha')$ belong to $\mathfrak{m}$. Up to the Weyl group actions,
there are only very few choices of $\alpha$ and $\beta$, and in each case, one  can easily find
a perpendicular pair of roots $\gamma_1$ and $\gamma_2$ from $\mathfrak{m}$. Then any nonzero
vectors $v_1\in\mathfrak{g}_{\pm\gamma_1}$ and $v_2\in\mathfrak{g}_{\pm\gamma_2}$ span
a FSCS, which is a contradiction.
%
%
%
%
In summarizing, the normal homogeneous Finsler space $(M,F)$ can not be positively curved in this case.

\medskip
\noindent\textbf{ The case $\mathfrak{g}=F_4$}.\quad
The standard presentation for the root system is given in (\ref{root-system-F}).
Upon the Weyl group action, we only need to consider the following subcases:

\medskip
\textbf{Subcase 1.}\quad The angle between $\alpha$ and $\beta$ is $\pi/4$. In this case, we can assume that $\alpha=e_1+e_2$
and $\beta=e_2$. Then $\alpha'=e_2$ is a root of $\mathfrak{h}$. By (1) of Lemma
\ref{trick-lemma-2}, $\pm e_1$ are roots of $\mathfrak{m}$. The length of the vector
$\mathrm{pr}(\frac{1}{2}e_1+\frac{1}{2}e_2+\frac{1}{2}e_3+\frac{1}{2}e_4)$ is
$\frac{\sqrt{3}}{2}$ and it is not orthogonal to $\alpha'$. So it is not a root
of $\mathfrak{h}$. Thus
$\pm(\pm\frac{1}{2}e_1+\frac{1}{2}e_2+\frac{1}{2}e_3+\frac{1}{2}e_4)$ are roots of
$\mathfrak{m}$. Then one can deduce a contradiction by applying (3) of Lemma \ref{trick-lemma-2} to $\pm e_1$ and $\pm(\pm\frac{1}{2}e_1+\frac{1}{2}e_2+\frac{1}{2}e_3+\frac{1}{2}e_4)$.



\medskip
\textbf{Subcase 2.}\quad  The angle between $\alpha$ and $\beta$ is $\frac{\pi}2$. Using suitable Weyl group actions one can reduce the problem  to the following cases:
\begin{description}
\item{\rm (1)} $\alpha$ and $\beta$ are long roots. Without losing generality, we can assume that  $\alpha=e_1+e_2$,
$\beta=e_2-e_1$;
\item{\rm (2)} $\alpha$ and $\beta$ are not of the same length. Without losing generality,  we can assume that
$\alpha=e_1+e_2$ and $\beta=-e_3$;
\item{\rm (3)} $\alpha$ and $\beta$ are short roots. Without losing generality,  we can assume that $\alpha=e_1$
and $\beta=e_2$.
\end{description}
Notice that (1) is covered  in the the case of the last subcase. We will discuss the rest situations in the following.

First consider the case $\alpha=e_1+e_2$ and $\beta=-e_3$. Then
$\alpha'=\frac{1}{3}e_1+\frac{1}{3}e_2-\frac{2}{3}e_3$ is a root of $\mathfrak{h}$ with   length
$\sqrt{\frac{2}{3}}$, and
$\hat{\mathfrak{g}}_{\pm\alpha'}=\mathfrak{g}_{\pm(e_1+e_2)}+\mathfrak{g}_{\pm e_3}$.
By (1) of Lemma \ref{trick-lemma-2}, $\pm e_4$ are roots of $\mathfrak{h}$, and
$\hat{\mathfrak{g}}_{\pm e_4}=\mathfrak{g}_{\pm e_4}$. Thus  $\mathfrak{g}_{\pm e_4}\subset
\mathfrak{h}$.
Now the vector $\mathrm{pr}(-e_3+e_4)$ is not orthogonal to $\alpha'$, and have  length
$\sqrt{\frac{5}{3}}$. So it is not a root of $\mathfrak{h}$, and
$\mathfrak{g}_{\pm(e_4-e_3)}\subset \mathfrak{m}$. Therefore we have
\begin{equation}
\mathfrak{g}_{\pm e_3}=[\mathfrak{g}_{\pm e_4},\mathfrak{\mathfrak{g}}_{\pm(e_4-e_3)}]
\subset \mathfrak{m},
\end{equation}
and $\mathfrak{h}_{\pm\alpha}=\mathfrak{g}_{\pm(e_1+e_2)}$, which is a contradiction with the  assumption.

Now we assume that
$\alpha=e_1$ and $\beta=e_2$. Then $\alpha'=\frac{1}{2}(e_1+e_2)$  is a root of $\mathfrak{h}$. By (2) of Lemma \ref{trick-lemma-2},
$e_1+e_2=2\alpha'$ is also a root of $\mathfrak{h}$, which is a contradiction.

\medskip
\textbf{Subcase 3.}\quad The angle between $\alpha$ and $\beta$ is $\frac{3\pi}{4}$. Without losing generality, we assume that $\alpha=e_1+e_2$
and $\beta=-e_2$. Then $\alpha'=\frac25 e_1-\frac15 e_2$ is a root of $\mathfrak{h}$ with
 length $\frac{1}{\sqrt{5}}$. There are only two roots of $\mathfrak{g}$ in
$\mathrm{pr}^{-1}(\alpha')$, i.e., $\mathfrak{h}_{\pm\alpha'}\subset
\hat{\mathfrak{g}}_{\pm\alpha'}=\mathfrak{g}_{\pm(e_1+e_2)}+\mathfrak{g}_{\pm e_2}$.
By (2) of Lemma \ref{trick-lemma-2}, $e_3$ is a root of $\mathfrak{h}$. It is obvious
that $e_3$ is the only root of $\mathfrak{g}$ in $\mathrm{pr}^{-1}(e_3)$, i.e.,
$\mathfrak{g}_{\pm e_3}=\mathfrak{h}_{\pm e_3}$. The vector $\mathrm{pr}(e_2+e_3)$ is
not a root of $\mathfrak{h}$, since it is not orthogonal to $\alpha'$ and its length
is $\frac{\sqrt{6}}{\sqrt{5}}$. So $e_2+e_3$ is a root of $\mathfrak{m}$, i.e.,
$\mathfrak{g}_{\pm(e_2+e_3)}\subset\mathfrak{m}$. So we have
\begin{equation}
\mathfrak{g}_{\pm e_2}\subset [\mathfrak{g}_{\pm (e_2+e_3)},\mathfrak{g}_{\pm e_3}]
\subset \mathfrak{m}.
\end{equation}
Then $\mathfrak{h}_{\pm\alpha'}$ must be the root plane $\mathfrak{g}_{\pm(e_1+e_2)}$,
which is a contradiction to our assumption.


To summarize, the normal homogeneous Finsler space $(M,F)$ can not be positively curved in this case.

\medskip
\noindent\textbf{The case $\mathfrak{g}=E_6$}.\quad  The root system can be identified with the subset
\begin{equation}\label{root-system-E6}
\{\pm e_i\pm e_j|\, 1\leq i<j<6\}\cup\{\pm\frac{1}{2}e_1\pm\cdots\pm\frac{1}{2}e_5\pm
\frac{\sqrt{3}}{2}e_6, \mbox{ with odd plus signs}\}
\end{equation}
in $\mathbb{R}^6$ with the standard orthonormal basis $\{e_1,\ldots,e_6\}$.
We only need to consider the case that
%
%
the angle between $\alpha$ and $\beta$ is $\frac{\pi}{2}$. Using suitable Weyl group actions, we can assume that $\alpha=e_1+e_2$ and
$\beta=-e_1+e_2$, or $\alpha=e_1+e_2$ and $\beta=e_3+e_4$.
However, it is easy to see that for $E_6$, we have more automorphisms for the root system, which can reduce the
discussion the case that $\alpha=e_1+e_2$ and $\beta=e_2-e_1$.
Then the unit vector $\alpha'=e_2$ is a root of $\mathfrak{h}$.
By calculating the orthogonal projections of all the roots of $\mathfrak{g}$ in
$\mathfrak{t}\cap\mathfrak{h}$, one easily sees that  none of the vectors  $\mathrm{pr}(\pm\frac{1}{2}e_1\pm\cdots\pm\frac{1}{2}e_5\pm
\frac{\sqrt{3}}{2}e_6)$ is a  root of $\mathfrak{h}$.  Thus the roots of the subalgebra
$\mathfrak{k}$ are those of
$D_5$, i.e., $\mathfrak{k}\neq\mathfrak{g}$. By the discussion in Section \ref{s6},
$G/K$ admits positively curved normal homogeneous Finsler metrics, which is impossible by Proposition \ref{summary-prop}.

\medskip
To summarize, the normal homogeneous Finsler space $(M,F)$ can not be positively curved in this case.

\medskip
\noindent\textbf{The case $\mathfrak{g}=E_7$}.\quad The root system can be identified with the subset
\begin{eqnarray}
& &\{\pm e_i\pm e_j, \forall 1\leq i<j\leq 6;\pm\sqrt{2}e_7\}\nonumber\\
& &\cup\{\pm\frac{1}{2}e_1\pm\cdots\pm\frac{1}{2}e_6
\pm\frac{\sqrt{2}}{2}e_7, \mbox{ where the number of }+\frac{1}{2}\mbox{'s} \mbox{ is even}\}\label{root-system-E7}
\end{eqnarray}
in $\mathbb{R}^7$ with the standard orthonormal basis $\{e_1,\ldots,e_7\}$.
%
We only need to consider the case that the angle between  $\alpha$ and $\beta$ is $\frac{\pi}{2}$.
Using suitable Weyl group actions we can  reduce the discussion to two situations that
$\alpha=e_1+e_2$ and $\beta=e_2-e_1$, or $\alpha=e_1+e_2$ and $\beta=-e_3-e_4$.

\medskip
\textbf{Subcase 1.}\quad Assume that $\alpha=e_1+e_2$, $\beta=e_2-e_1$. Then $\alpha'=e_2$ is
 a root of $\mathfrak{h}$. By (2) of Lemma \ref{trick-lemma-2}, $\pm\sqrt{2}e_7$ are roots of $\mathfrak{h}$, and obviously $\mathfrak{h}_{\pm\sqrt{2}e_7}=\mathfrak{g}_{\pm\sqrt{2}e_7}$. On the other hand, none of the vectors among $\mathrm{pr}(\pm\frac{1}{2}e_1\pm\cdots\pm\frac{1}{2}e_6
\pm\frac{\sqrt{2}}{2}e_7)$ is orthogonal to $\alpha'$ or has the proper length.
So except $\pm\sqrt{2}e_7$, all roots of $\mathfrak{h}$ are linear combinations of the first six $e_i$s.
Then the subalgebra $\mathfrak{k}$ is
a subalgebra of $D_6\oplus A_1$, which is a proper subalgebra of $\mathfrak{g}$. By the discussion in Section \ref{s6},
$G/K$ admits positively curved normal homogeneous Finsler metrics, which is impossible by Proposition \ref{summary-prop}.

\medskip
\textbf{Subcase 2.}\quad Assume that $\alpha=e_1+e_2$ and $\beta=-e_3-e_4$. Then the unit vector $\alpha'=\frac{1}{2}e_1+\frac{1}{2}e_2-\frac{1}{2}e_3-\frac{1}{2}e_4$ is a root of
$\mathfrak{h}$. Thus none of the vectors
$\mathrm{pr}(\pm e_i\pm e_j)$, with $1\leq i\leq 4<j<7$, or
$\mathrm{pr}(\pm(\frac{1}{2}e_1+\frac{1}{2}e_2+\frac{1}{2}e_3+\frac{1}{2}e_4-e_i)\pm
(\frac{1}{2}e_5-\frac{1}{2}e_6)\pm\frac{\sqrt{2}}{2}e_7)$, with  $1\leq i\leq 4$,
is  orthogonal to $\alpha'$. Moreover,   they all have  length $\frac{\sqrt{7}}{2}$. Thus  they
are not roots of $\mathfrak{h}$.
By (2) of Lemma \ref{trick-lemma-2}, $\pm(e_5-e_6)$ are
roots of $\mathfrak{h}$. It is obvious that $\hat{\mathfrak{g}}_{\pm(e_5-e_6)}=\mathfrak{g}_{\pm(e_5-e_6)}$.
 Now by the above discussion, for any root $\gamma'\neq\pm(e_5-e_6)$ of $\mathfrak{h}$, and any root plane
$\mathfrak{g}_{\pm\gamma}$ in $\hat{\mathfrak{g}}_{\pm\gamma'}=\sum_{\mathrm{pr}(\gamma)=\gamma'}\mathfrak{g}_{\gamma}$,
$\gamma$ is orthogonal to $e_5-e_6$, and
$[\mathfrak{g}_{\pm\gamma},\mathfrak{g}_{\pm(e_5-e_6)}]=0$.
Then the subalgebra $\mathfrak{k}$ is contained in the subalgebra $\mathfrak{k}'$ of $\mathfrak{g}$ generated by $\mathfrak{t}$, and the root subspaces $\mathfrak{g}_{\pm\gamma}$ among the subspaces  $\hat{\mathfrak{g}}_{\pm\gamma'}$. Thus $\mathfrak{k}'$ has a three dimensional
ideal (an $A_1$) spanned by $\mathfrak{g}_{\pm(e_5-e_6)}$ and $e_5-e_6\in\mathfrak{t}$.
In particular, $\mathfrak{k}\subset\mathfrak{k}'\neq\mathfrak{g}$.
By the discussion in Section \ref{s6},
$G/K$ admits positively curved normal homogeneous Finsler metrics, which is impossible by Proposition \ref{summary-prop}.

To summarize, the normal homogeneous Finsler space $(M,F)$ is not positively curved in this case.

\medskip
\noindent\textbf{The case $\mathfrak{g}=E_8$}.\quad The root system can be identified with the subset
\begin{equation}
\{\pm e_i\pm e_j|\, 1\leq i<j\leq 8\}\cup\{\pm\frac{1}{2}e_1\pm\cdots\pm
\frac{1}{2}e_8 \mbox{ with  even number of plus signs}\}
\end{equation}
in $\mathbb{R}^8$ with the standard orthonormal basis
$\{e_1,\ldots,e_8\}$.
%
%
We only need to consider the case that the angle between $\alpha$ and $\beta$ is  $\frac{\pi}{2}$.
Using suitable Weyl group actions we can reduce the discussion to the two situations that
$\alpha=e_1+e_2$ and $\beta=e_2-e_1$, or $\alpha=e_1+e_2$ and $\beta=-e_3-e_4$.

\medskip
\textbf{Subcase 1.}\quad
Assume that $\alpha=e_1+e_2$ and $\beta=e_2-e_1$. Then the unit vector $\alpha'=e_2$ is a root
of $\mathfrak{h}$. Note that none of the  vectors $\mathrm{pr}(\pm\frac{1}{2}e_1\pm\cdots\pm
\frac{1}{2}e_8)$ is a root of $\mathfrak{h}$, since they
have the same length $\frac{\sqrt{7}}{2}$, and  are not
orthogonal to $\alpha'$.
Thus  $\mathfrak{k}$ is contained in the subalgebra $\mathfrak{k}'$  generated by $\mathfrak{t}$
and all root planes in the root subspaces $\hat{\mathfrak{g}}_{\pm\gamma'}$. Therefore $\mathfrak{k}'$ is contained in a subalgebra   generated by $\mathfrak{t}$ and the root subspaces
$\mathfrak{g}_{\pm(e_i-e_i)}$, which is isomorphic to $D_8$. In particular,  $\mathfrak{k}\neq\mathfrak{g}$.
By the discussion in Section \ref{s6},
$G/K$ admits positively curved normal homogeneous Finsler metrics, which is impossible by Proposition \ref{summary-prop}.

\medskip
\textbf{Subcase 2.}\quad Assume  that $\alpha=e_1+e_2$ and $\beta=-e_3-e_4$. Then the unit vector
$\alpha'=\frac12(e_1+e_2-e_3 -e_4)$ is a root of $\mathfrak{h}$.
Note that none of the vectors $\mathrm{pr}(\pm e_i\pm e_j)$ with $1\leq i\leq 4<j\leq 8$ and
$\mathrm{pr}(\frac12 (\pm e_1\pm \cdots\pm e_8))$, where there are odd number of plus signs in the first four terms and the rest four terms, is  a root
of $\mathfrak{h}$, since they are not orthogonal to $\alpha'$ and they have the same
improper length $\frac{\sqrt{7}}{2}$. So the root planes of the corresponding roots
of $\mathfrak{g}$ are all contained in $\mathfrak{m}$. Now a direct calculation shows that the orthogonal complement of the
sum of these root planes is in fact the subalgebra $\mathfrak{k}'$ linearly spanned by
$\mathfrak{t}$, the root planes of $\mathfrak{g}$ for the roots $\pm e_i\pm e_j$,
$1\leq i<j\leq 4$ or $5\leq i<j\leq 8$, and $\frac{1}{2}(\pm e_1\pm\cdots\pm e_8)$, where there
are even numbers of plus signs in the first four terms and in the rest four terms.
This subalgebra $\mathfrak{k}'$ is another $D_8$ in $\mathfrak{g}$. It contains $\mathfrak{h}$ and $\mathfrak{t}$, hence it
contains $\mathfrak{k}$, which is not equal to $\mathfrak{g}$.
By the discussion in Section \ref{s6},
$G/K$ admits positively curved normal homogeneous Finsler metrics, which is impossible by Proposition \ref{summary-prop}.

\medskip
To summarize, the normal Finsler space $(M,F)$ cannot be positively curved in this case.

\subsection{The $A$ and $D$ cases}
We will  use the standard presentation of the root systems of classical types given in Section \ref{s6}.

\medskip
\noindent\textbf{The case $\mathfrak{g}=D_n$, $n>3$}.
%
%
We only need to consider the case that the angle between  $\alpha$ and $\beta$ is $\frac{\pi}{2}$.
Usng suitable automorphism of $\mathfrak{g}$ we can reduce the problem to the two situations that
 $\alpha=e_1+e_2$ and $\beta=e_2-e_1$, or $\alpha=e_1+e_2$ and $\beta=-e_3-e_4$. Note that the last case can  appear only  when $n>4$.

\medskip
\textbf{Subcase 1.}\quad Assume that $\alpha=e_1+e_2$ and $\beta=e_2-e_1$. Then $\alpha'=e_2$ is a root of $\mathfrak{h}$ with  length $1$. We assert that  the subalgebra $\mathfrak{h}$ must be isomorphic $B_{n-1}$. To
see this, we need to determine its root systems.
By (2) of Lemma \ref{trick-lemma-2}, all the roots $\pm e_i\pm e_j$ of $\mathfrak{g}$, with $1<i<j\leq n$, are also roots
of $\mathfrak{h}$. It is obvious that
$\mathfrak{h}_{\pm e_i\pm e_j}=\mathfrak{g}_{\pm e_i\pm e_j}=
\hat{\mathfrak{g}}_{\pm e_i\pm e_j},$
when $1<i<j\leq n$. By Lemma \ref{trick-lemma-0}, for any nonzero vector
$v\in\mathfrak{g}_{\pm(e_2-e_i)}$ with $i>2$, $\mathrm{ad}(v)$ maps $\hat{\mathfrak{g}}_{\pm e_2}=\mathop{\sum}\limits_{\varepsilon=\pm 1}\mathfrak{g}_{\pm(e_2+\varepsilon e_1)}$ isomorphically onto
$\mathop{\sum}\limits_{\varepsilon=\pm 1}\mathfrak{g}_{\pm(e_i+\varepsilon e_1)}$, and it preserves the decomposition $\mathfrak{g}=\mathfrak{h}+\mathfrak{m}$.
Thus for each $i\geq 2$, $e_i$ is a root of $\mathfrak{h}$, and
$\hat{\mathfrak{g}}_{\pm e_i}=\mathop{\sum}\limits_{\varepsilon=\pm 1}\mathfrak{g}_{\pm(e_i+\varepsilon e_1)}$. This argument has exhausted all possible
roots of $\mathfrak{h}$ from the projections of roots of $\mathfrak{g}$. Hence
the root system of $\mathfrak{h}$ is  $B_{n-1}$ (see (\ref{root-system-B})).

From the above arguments, we see that $\mathfrak{h}$ is totally determined by the choice of
$\mathfrak{h}_{\pm e_2}$ in $\hat{\mathfrak{g}}_{\pm e_2}=
\mathop{\sum}\limits_{\varepsilon=\pm 1}\mathfrak{g}_{\pm(e_2+\varepsilon e_1)}$.
We only need to discuss the problem within the subalgebra
$\mathfrak{g}'=\mathbb{R}e_1+\mathbb{R}e_2+\mathfrak{g}_{\pm(e_1+e_2)}+
\mathfrak{g}_{\pm(e_2-e_1)}$, which is isomorphic to $A_1\oplus A_1$, where  the first (resp. second) $A_1$ factor is algebraically
generated by
$\mathfrak{g}_{\pm(e_1+e_2)}$ (resp. $\mathfrak{g}_{\pm(e_2-e_1)}$). Now we will prove a  lemma showing that a suitable $\mathrm{Ad}(\exp(\mathbb{R}e_1+\mathbb{R}e_2))$-action, which preserves $\mathfrak{t}$ and all the roots,
gives an isomorphism between any of the two possible subalgebras $\mathfrak{h}$.

We first give a definition. For a compact Lie algebra of  type $A_1$ endowed with a bi-invariant metric, we call an orthogonal basis $\{u_1,u_2,u_3\}$ standard, if $u_1,u_2,u_3$  have the same length, and they satisfying the condition $[u_i,u_j]=u_k$ for $(i,j,k)=(1,2,3)$, $(2,3,1)$ or $(3,1,2)$. The length $c$ of the vectors in a standard basis is a
constant which only depends on the the scale of the bi-invariant inner product. In fact, the bracket of any two
orthogonal vectors with  length $c$ is also a vector with a length $c$.

\begin{lemma}\label{lemma-double-A_1}
Let $\mathfrak{g}'=\mathfrak{g}_1\oplus\mathfrak{g}_2=A_1\oplus A_1$ be
endowed with a bi-invariant inner product. Assume that $\mathfrak{t}'$ is a Cartan
subalgebra, and $\mathfrak{h}'$ and $\mathfrak{h}''$ are subalgebras isomorphic to $A_1$
satisfying the following conditions:
\begin{description}
\item{\rm (1)} $\mathfrak{h}'\cap\mathfrak{t}'=\mathfrak{h}''\cap\mathfrak{t}'$ is one
dimensional;
\item{\rm (2)} $\mathfrak{h}'\cap\mathfrak{g}_i=\mathfrak{h}''\cap\mathfrak{g}_i=0$, $i=1,2$.
\item{\rm (3)} $\mathfrak{h}'\cap(\mathfrak{h'}\cap\mathfrak{t}')^{\perp}\subset
\mathfrak{t}'^{\perp}$, and $\mathfrak{h}''\cap(\mathfrak{h''}\cap\mathfrak{t}')^{\perp}\subset
\mathfrak{t}'^{\perp}$.
\end{description}
Then there is an $\mathrm{Ad}(\exp \mathfrak{t}')$-action which maps $\mathfrak{h}'$
to $\mathfrak{h}''$.
\end{lemma}

\begin{proof} Let $c_1$ and $c_2$ be the length of standard basis vectors for $\mathfrak{g}_1$ and $\mathfrak{g}_2$, respectively. We can choose standard bases $\{u_1,u_2,u_3\}$ and $\{v_1,v_2,v_3\}$
for $\mathfrak{g}_1$ and $\mathfrak{g}_2$ as follows. First, we choose vectors
$u_1$ and $v_1$ from $\mathfrak{t}'\cap\mathfrak{g}_1$ and $\mathfrak{t}'\cap\mathfrak{g}_2$ with lengths $c_1$ and $c_2$, respectively. Then we freely choose any vector $u_2$ of length $c_1$ from
$\mathfrak{t}'^\perp\cap\mathfrak{g}_1$ and then set $u_3=[u_1,u_2]$. By (2) and
(3) in the lemma, we can find a vector of
$\mathfrak{h}$ which belongs to $ u_2+\mathfrak{g}_2\cap\mathfrak{t}'^{\perp}$. Then its
$\mathfrak{g}_2$-factor is not $0$, which can be positively scaled to have  length $c_2$, and we can set this vector to be $v_2$. Finally, we set
$v_3=[v_1,v_2]$.
Now the subalgebra $\mathfrak{h}'$ is linearly spanned by $u_1+av_1$, $u_2+bv_2$, and their
bracket $u_3+abv_3$, where $a$ is a fixed nonzero constant and $b>0$. However, as a subalgebra, it must satisfy $[u_2+bv_2,u_3+abv_3]=u_1+av_1$, hence $b=1$.
We can similarly construct another standard basis $\{v'_1,v'_2,v'_3\}$ for $\mathfrak{h}''$. Notice that $v'_1=v_1$. It is easy to see that there exists  a real number $t$ such that $\mathrm{Ad}(\exp(tv_1))$ maps
$v_2$ to $v'_2$ and $v_3$ to $v'_3$. Then the above calculation indicates that  it also maps
$\mathfrak{h}'$ to $\mathfrak{h}''$.
\end{proof}

In conclusion,in this case $(M, F)$ must be equivalent to a symmetric sphere $S^{2n-1}=\mathrm{SO}(2n)/\mathrm{SO}(2n-1)$,
with $\mathfrak{g}=D_n$ and $\mathfrak{h}=B_{n-1}$. The argument is also valid for $\mathfrak{g}=D_3=A_3$.

\medskip
\textbf{Subcase 2.}\quad Assume that $n>4$, and $\alpha=e_1+e_2$, $\beta=-e_3-e_4$. Then $\alpha'=e_1+e_2-e_3-e_4$
is a root of $\mathfrak{h}$. None of the vectors $\mathrm{pr}(\pm e_i\pm e_j)$ with
$1\leq i\leq 4 < j\leq n$ is orthogonal to $\alpha'$ and they all have the same
length $\frac{\sqrt{7}}{2}$. By similar arguments as before, any root of the subalgebra $\mathfrak{k}$ is either a linear combination of the elements $e_i$ with $i\leq 4$, or a linear combination of the elements $e_i$ with $i>4$.
Thus $\mathfrak{k}\neq\mathfrak{g}$ and $G/K$ have positive curved normal homogeneous
Finsler metrics. But this is a contradiction with Proposition \ref{summary-prop}.

To summarize, in this case, the positively curved normal homogeneous space $(M,F)$ must be equivalent
to the symmetric sphere $S^{2n-1}=\mathrm{SO}(2n)/\mathrm{SO}(2n-1)$, where $n>3$.

\medskip
\noindent\textbf{The case $\mathfrak{g}=A_n$}.
%
%
We only need to consider the case that the angle between  $\alpha$ and $\beta$ is $\frac{\pi}{2}$ (hence $n>2$). By suitable Weyl group actions, we can assume that $\alpha=e_1-e_2$ and
$\beta=e_3-e_4$. Then the unit vector $\alpha'=\frac12 e_1-\frac12 e_2+\frac12 e_3-\frac12 e_4$ is a root of $\mathfrak{h}$.

If $n>4$, then none of the vectors $\mathrm{pr}(\pm(e_i-e_j))$, with
$1\leq i\leq 4<j\leq n+1$, is orthogonal to $\mathfrak{h}$, and they all have the same  length
$\frac{\sqrt{7}}{2}$. Thus none of the above vectors is  a root of $\mathfrak{h}$. On the other hand, a root of $\mathfrak{h}$
 is either a linear combinations of  the vectors $e_i$ with $i\leq 4$, or a linear combination of
 the vectors $e_j$ with $j>4$. The same assertion also holds for any root of the subalgebra $\mathfrak{k}$. Hence $\mathfrak{k}$ is
 contained in a subalgebra of the type $A_3\oplus A_{n-4}\oplus\mathbb{R}\in\mathfrak{g}$.
 But the pair $(\mathfrak{g},\mathfrak{k})$ is not covered in (2) of Proposition \ref{summary-prop}. This is a contradiction.
 Thus in this case, the corresponding
 normal homogeneous Finsler space $(M, F)$ can not be positively curved.

 If $n=3$, then $\mathfrak{g}$ can also be viewed as $D_3$, with $\alpha=e_1+e_2$ and $\beta=e_1-e_2$. We have shown in the above that the positively curved normal homogeneous Finsler space $(M,F)$ is equivalent to the symmetric sphere $S^5=\mathrm{SO}(6)/\mathrm{SO}(5)$.

Suppose $n=4$. Let $\mathfrak{g}'$ be the subalgebra $A_3$ corresponding to the first four
roots $e_i$, $1\leq i\leq 4$. By the arguments for $n>4$, one easily sees that all the roots of $\mathfrak{k}$, which is generated by $\mathfrak{h}$ and $\mathfrak{t}$, must belong to $\mathfrak{g}'$.
Then by Proposition \ref{summary-prop}, we have $\mathfrak{k}=\mathfrak{g}'\oplus\mathbb{R}$.

By similar arguments as in Subcase 2 for $D_n$, we can see $\mathfrak{h}=\mathfrak{h}'\oplus\mathbb{R}$, in which $\mathfrak{h}'$ is
a subalgebra of $\mathfrak{g}'$ isomorphic to $B_2=C_2$ with the following roots
$$\{\pm (e_1-e_4),\pm(e_2-e_3),\pm(\frac12 e_1-\frac12 e_2+\frac12 e_3-\frac12 e_4),\pm(\frac12 e_1+\frac12 e_2-\frac12 e_3-\frac12 e_4)\}.$$
Moreover,  up to inner isomorphisms of $\mathfrak{g}'$, the subalgebra $\mathfrak{h}'$ is uniquely determined. Berger \cite{Ber61} has proved that, in this case,  $\mathrm{SU}(5)/\mathrm{Sp}(2)S^1$ admits a positively curved
normal homogeneous Riemannian metric. It is generally called a Berger's space.
Another Berger's space $\mathrm{Sp}(2)/\mathrm{SU}(2)$ will appear later.

\medskip
To summarize, in this case, $(M,F)$ is positively curved if and only if  it is equivalent
to the symmetric normal homogeneous sphere $S^5=\mathrm{SO}(6)/\mathrm{SO}(5)$, or
the  Berger's space $\mathrm{SU}(5)/\mathrm{Sp}(2)S^1$.

\subsection{The $B$ and $C$ cases}

We keep all the notations as above and use the standard presentations
for the root systems as in the last section.

\medskip
\noindent\textbf{The case $\mathfrak{g}=B_n$, $n>1$}.
Using suitable Weyl group actions we can reduce the the consideration to
 the following cases.

\medskip
\textbf{Subcase 1.}\quad The angle between $\alpha$ and $\beta$ is $\frac{\pi}{4}$. In this case, we can assume that $\alpha=e_1+e_2$ and $\beta=e_2$. Then $\alpha'=e_2$ is a root of $\mathfrak{h}$ with
$\hat{\mathfrak{h}}_{\pm e_2}=\mathfrak{g}_{\pm(e_2-\epsilon e_1)}+\mathfrak{g}_{\pm e_2}+\mathfrak{g}_{\pm(e_2+e_1)}$.
%
Denote $\mathfrak{g}'=\mathbb{R}e_1+\mathbb{R}e_2+\sum_{a,b}
\mathfrak{g}_{\pm(a e_1+b e_2)}$ and $\mathfrak{g}''=\mathbb{R}e_1+\mathfrak{g}_{\pm e_1}$. They are Lie algebras $B_2$ and $A_1$ respectively.
We can use real matrices in $\mathfrak{so}(5)$ to give a basis of $\mathfrak{h}\cap\mathfrak{g}'$, i.e.,
$$u=\left(
      \begin{array}{ccccc}
        0 & 0 & 0 & 0 & 0 \\
        0 & 0 & 0 & 0 & 0 \\
        0 & 0 & 0 & 0 & 0 \\
        0 & 0 & 0 & 0 & -1 \\
        0 & 0 & 0 & 1 & 0\\
      \end{array}
    \right)\in\mathbb{R}e_2,\quad
v=\left(
    \begin{array}{ccccc}
      0 & 0 & 0 & -a & -a' \\
      0 & 0 & 0 & -b & -b' \\
      0 & 0 & 0 & -c & -c' \\
      a & b & c & 0 & 0 \\
      a' & b' & c' & 0 & 0 \\
    \end{array}
  \right)\in\mathfrak{h}_{\pm e_2}
$$
and
$$w=[u,v]=\left(
    \begin{array}{ccccc}
      0 & 0 & 0 & a' & -a \\
      0 & 0 & 0 & b' & -b \\
      0 & 0 & 0 & c' & -c \\
      -a' & -b' & -c' & 0 & 0 \\
      a & b & c & 0 & 0 \\
    \end{array}
  \right)\in\mathfrak{h}_{\pm e_2}.
$$
 Since $\mathfrak{h}\cap\mathfrak{g}'$ is a Lie subalgebra, we have $[v,w]\in\mathfrak{h}\cap\mathfrak{g}'$. A direct calculation then shows that
$(a,b,c)$ and $(a',b',c')$ are linearly dependent vectors. Using a suitable element
$l\in\mathrm{Ad}(\exp A_1)$, which is represented by a conjugation by a matrix of the form
$\mathrm{diag}(P,I)$, where $P\in\mathrm{SO}(3)$, and $I$ is the $2\times 2$ identity
matrix, we can make $b=c=b'=c'=0$ in $v$. Let $u'\in\mathfrak{g}_{\pm(e_2+e_1)}$ and
$v'\in\mathfrak{g}_{\pm(e_2-e_1)}$ be a pair of nonzero vectors. Then $l^{-1}(u')$ and
$l^{-1}(v')$ are vectors of $\mathfrak{m}\cap\mathfrak{g}'$. Since $[u', v']=0$, we have $[l^{-1}(u'),l^{-1}(v')]=0$. Now $l^{-1}(u')$ and $l^{-1}(v')$, together with all the elements $e_i$ with  $i>2$,
span a FSCS. This is a contradiction.

To summarize, there is no positively curved normal homogeneous Finsler space in this subcase.

%

\medskip
In the following subcases we  assume that the angle between $\alpha$ and $\beta$ is $\frac{\pi}2$.
Then using suitable Weyl group actions one can reduce the discussion to one of the following situations:
\begin{description}
\item{\rm (1)} $\alpha$ and $\beta$ are long roots. In this case, we can assume that either $\alpha=e_1+e_2$ and $\beta=e_2-e_1$, or $\alpha=e_1+e_2$ and $\beta=-e_3-e_4$;
\item{\rm (2)} $\alpha$ and $\beta$ are not of the same length.  Then we can assume
    $\alpha=e_1+e_2$ and $\beta=-e_3$;
\item{\rm (3)} $\alpha$ and $\beta$ are short roots. Then we can assume that $\alpha=e_1$
and $\beta=e_2$.
\end{description}
We now consider these situations case by case.

\medskip
\textbf{Subcase 2.}\quad The situation that $\alpha=e_1+e_2$ and $\beta=e_2-e_1$ has
been covered by the discussion in the last subcase, and in this case there does not exist any positively curved normal homogeneous spaces.

Now we assume that
$\alpha=e_1+e_2$ and $\beta=-e_3-e_4$. Then the unit vector
$\alpha'=\frac12 e_1+\frac12 e_2-\frac12 e_3-\frac12 e_4$ is a root of $\mathfrak{h}$.
If $n>4$, then none of the vectors $\pm e_i\pm e_j$ with $1\leq i\leq 4<j\leq n$ is  orthogonal to
$\alpha'$ and each of them has the same length $\frac{\sqrt{7}}{2}$. therefore none of the above roots  it is  a root of $\mathfrak{h}$.
The subalgebra $\mathfrak{k}$ is then
contained in the subalgebra $B_4\oplus B_{n-4}$ or $B_4\oplus A_1$. The coset space $G/K$
can admit positively curved normal homogeneous Finsler metrics. This is a contradiction with Proposition \ref{summary-prop}. Hence in this case the normal homogeneous space cannot be positively curved.

If $n=4$, then by (2) of Lemma \ref{trick-lemma-2}, all the roots $\pm (e_i-e_j)$ of $\mathfrak{g}$, with $1\leq i<j\leq 4$,  are roots of $\mathfrak{h}$, and we have
$\mathfrak{h}_{\pm(e_i-e_j)}=\mathfrak{g}_{\pm(e_i-e_j)}=\hat{\mathfrak{g}}_{\pm(e_i-e_j)}$
for $1\leq i<j\leq 4$. For any $v\in\mathfrak{g}_{\pm(e_2-e_3)}$, $\mathrm{ad}(v)$
maps $\hat{\mathfrak{g}}_{\pm\frac12(e_1+e_2-e_3-e_4)}=\mathfrak{g}_{\pm(e_1+e_2)}+
\mathfrak{g}_{\pm(e_3+e_4)}$ isomorphically onto
$\mathfrak{g}_{\pm(e_1+e_3)}+\mathfrak{g}_{\pm(e_2+e_4)}$, which preserves the decomposition $\mathfrak{g}=\mathfrak{h}+\mathfrak{m}$. So $\pm\frac12(e_1+e_3-e_2-e_4)$ is also a root of $\mathfrak{h}$, and $\hat{\mathfrak{g}}_{\pm\frac12(e_1+e_3-e_2-e_4)}=
\mathfrak{g}_{\pm(e_1+e_3)}+\mathfrak{g}_{\pm(e_2+e_4)}$. The same assertion also holds for $\pm\frac12(e_1+e_4-e_2-e_3)$.
On the other hand,
 none of the  vectors
$\mathrm{pr}(\pm e_i)$ is  orthogonal to the unit root $\alpha'$ of $\mathfrak{h}$, and they
all have the same
length $\frac{\sqrt{3}}{2}$. Thus they are not roots of $\mathfrak{h}$.
In the above, we have determined all the roots of $\mathfrak{h}$ and showed that $\mathfrak{h}=B_3$.

There is a known example in this subcase, which provide the homogeneous sphere
$S^{15}=\mathrm{Spin}(9)/\mathrm{Spin}(7)$. Up to equivalence, this is the only one normal coset space in
this subcase. The subalgebra $\mathfrak{h}$ is totally determined by the choice of
$\mathfrak{h}_{\pm(e_1+e_2-e_3-e_4)}$. Consider the subalgebra
$\mathfrak{g}'=\mathbb{R}(e_1+e_2)+\mathbb{R}(e_3+e_4)+\mathfrak{g}_{\pm(e_1+e_2)}+
\mathfrak{g}_{\pm(e_3+e_4)}$ which is an $A_1\oplus A_1$ containing
$\hat{\mathfrak{g}}_{\pm(e_1+e_2-e_3-e_4)}$. By Lemma \ref{lemma-double-A_1},
$\mathfrak{h}\cap\mathfrak{g}'$, as well as $\mathfrak{h}_{\pm(e_1+e_2-e_3-e_4)}$,
is uniquely determined up to inner isomorphisms $\mathrm{Ad}(\exp(t(e_3+e_4)))$.
Note that in the paper \cite{VZ09}, the authors present the $\mathrm{Spin}(9)$-invariant
 Riemannian metric on $S^{15}$ as a family of Riemannian metrics
$g_t$ (up to homothety), and show that the Riemannian metric has positive curvature if and only if $0<t<\frac43$.
The normal Riemanniam metric is the metric corresponding to $t=\frac12$.

Therefore,  in this case, $(M,F)$ is
equivalent to the normal homogeneous sphere $S^{15}=\mathrm{Spin}(9)/\mathrm{Spin}(7)$.

\medskip
\textbf{Subcase 3.}\quad We assume that $\alpha=e_1+e_2$ and $\beta=-e_3$. Then
$\alpha'=\frac13 e_1+\frac13 e_2-\frac23 e_3$ is a root of $\mathfrak{h}$ with
 length $\frac{\sqrt{2}}{\sqrt{3}}$.
If $n>3$, then none of the vectors $\mathrm{pr}(\pm e_i\pm e_j)$, with
 $1\leq i\leq 3<j\leq n$,  is  orthogonal to $\alpha'$, and they all have the same  length $\sqrt{\frac53}$. Therefore none of the above roots is a  root of $\mathfrak{h}$. So any root of the subalgebra $\mathfrak{k}$ is  either a   linear combinations of $e_1$, $e_2$ and $e_3$, or  a   linear combinations of the elements $e_i$ with $i>3$. In particular, $\mathfrak{k}\neq\mathfrak{g}$, and the corresponding coset space $G/K$ admits positively curved normal homogeneous Finsler metrics. By Proposition \ref{summary-prop}, it must be
isomorphic to $D_n$, which is a contradiction. So the corresponding normal homogeneous Finsler space cannot be positively curved in this case.

If $n=3$, then $\mathfrak{h}=G_2$. In fact, by
(2) of Lemma \ref{trick-lemma-2}, $\pm (e_i-e_j)$ for $1\leq i<j\leq 3$
are long roots of $\mathfrak{h}$, with $\mathfrak{h}_{\pm(e_i-e_j)}=\mathfrak{g}_{\pm(e_i-e_j)}$. Then for any nonzero vector $v\in\mathfrak{g}_{\pm(e_2-e_3)}$, the linear map $\mathrm{ad}(v)$ sends  $\hat{\mathfrak{g}}_{\pm\alpha'}=\mathfrak{g}_{\pm(e_1+e_2)}+\mathfrak{g}_{\pm e_3}$
isomorphically onto $\mathfrak{g}_{\pm(e_1+e_3)}+\mathfrak{g}_{\pm e_2}$, and preserves the decomposition
$\mathfrak{g}=\mathfrak{h}+\mathfrak{m}$. Thus $\pm(\frac13 e_1-\frac23 e_2+\frac13 e_3)$
are roots of $\mathfrak{h}$, with
$\hat{\mathfrak{g}}_{\pm(\frac13 e_1-\frac23 e_2+\frac13 e_3)}=\mathfrak{g}_{\pm e_2}+\mathfrak{g}_{\pm(e_1+e_3)}$. The same assertion also holds for
$\pm(-\frac23 e_1+\frac13 e_2+\frac13 e_3)$. We have found all six long roots and six
short roots for $\mathfrak{h}$ and they exhaust all possible roots of $\mathfrak{h}$.
This implies that $\mathfrak{h}=G_2$.

The subalgebra $\mathfrak{h}$ is totally determined by the choice of $\mathfrak{h}_{\pm\alpha'}\subset\hat{\mathfrak{g}}_{\pm\alpha'}$.
Denote $\mathfrak{g}'=\mathbb{R}(e_1+e_2)+\mathbb{R}e_3+\mathfrak{g}_{\pm(e_1+e_2)}
+\mathfrak{g}_{\pm e_3}$, which is a  subalgebra of the type $A_1\oplus A_1$ containing
$\hat{\mathfrak{g}}_{\pm\alpha'}$.
By Lemma \ref{lemma-double-A_1}, we have $\mathfrak{h}\cap\mathfrak{g}'$. Meanwhile,
$\mathfrak{h}_{\pm\alpha'}$ is uniquely determined up to inner isomorphisms. So up to
equivalence, the only positively curved normal homogeneous Finsler space in this subcase is the homogeneous sphere $S^7=\mathrm{Spin}(7)/\mathrm{G}_2$.

\medskip
\textbf{Subcase 4.}\quad We assume that $\alpha=e_1$ and $\beta=e_2$. Then
$\alpha'=\frac12 e_1+\frac12 e_2$ is a root of $\mathfrak{h}$. But by (2) of Lemma
\ref{trick-lemma-2}, $e_1+e_2=2\alpha'$ is also a root of $\mathfrak{h}$. This
is a contradiction. So there does not exist positively curved normal homogeneous space in
this subcase.

\medskip
\textbf{Subcase 5.}\quad There is one more subcase left where the angle between $\alpha$ and $\beta$ is $\frac{3\pi}{4}$. We can assume that $\alpha=e_1+e_2$ and
$\beta=-e_2$. Then $\alpha'=\frac25 e_1-\frac15 e_2$ is a root of $\mathfrak{h}$ with  length
$\frac1{\sqrt{5}}$. If $n>2$, then none of the  vectors
$\mathrm{pr}(\pm e_i\pm e_j)$, with  $1\leq i\leq 2<j\leq n$,  is  orthogonal to $\alpha'$
and  each of them has  length $\frac{\sqrt{6}}{\sqrt{5}}$ or $\frac{\sqrt{7}}{\sqrt{5}}$. Thus none of these vectors can  be a root of $\mathfrak{h}$.
On the other hand, each root of the subalgebra $\mathfrak{k}$ is either a linear combination of $e_1$ and $e_2$, or a linear combination of the elements $e_i$, with $i>2$.
So $\mathfrak{k}$ is contained in $B_2\oplus B_{n-2}$ or $B_2\oplus A_1$. In particular, $\mathfrak{k}\neq\mathfrak{g}$ and the corresponding  coset space $G/K$ admits positively curved normal homogeneous metrics. However, according to Proposition \ref{summary-prop},
$G/K$ can not be positively curved normal homogeneous, which is a contradiction. This implies that in this subcase $(M, F)$ cannot be positively curved.

If $n=2$, then any linearly independent commuting pair in $\mathfrak{m}$ span a FSCS. Thus in the decomposition of the Lie algebra of a
positively curved normal homogeneous Finsler space, the subspace $\mathfrak{m}$  can not have a commuting pair. However, this  implies exactly that  the corresponding normal homogeneous Riemannian metric is positive curved.
Up to equivalence, there is only one such space, that is, the coset space $\mathrm{Sp}(2)/\mathrm{SU}(2)$ found by
Berger in \cite{Ber61}.

\medskip
To summarize, in this case, the positively curved normal homogeneous space $(M,F)$ is equivalent to $S^7=\mathrm{Spin}(7)/G_2$, $S^{15}=\mathrm{Spin}(9)/\mathrm{Spin}(7)$,
or the Berger's space $\mathrm{Sp}(2)/\mathrm{SU}(2)$,.

\medskip
\noindent\textbf{The case $\mathfrak{g}=C_n$, $n>2$}. Upon the Weyl group
action, we only need to consider
the following subcases.

\medskip
\textbf{Subcase 1.}\quad The angle between  $\alpha$ and $\beta$ is  $\frac{\pi}{4}$. We can assume that $\alpha=2e_1$ and $\beta=e_1+e_2$. Then $\alpha'=e_1+e_2$ is a root of $\mathfrak{h}$ with  length
$\sqrt{2}$. By (1) of Lemma \ref{trick-lemma-2},
we have $\mathfrak{g}_{\pm (e_1-e_2)}\subset\mathfrak{m}$. Now none of the vectors $\mathrm{pr}(\pm(e_1-e_3))=\mathrm{pr}(\pm(e_2-e_3))$ is  orthogonal to $\alpha'$,
and each has  length $\frac{\sqrt{3}}{\sqrt{2}}$. Thus none of them  is a root of $\mathfrak{h}$, and
we have $\mathfrak{g}_{\pm(e_1-e_3)}$, $\mathfrak{g}_{\pm(e_2-e_3)}\subset\mathfrak{m}$. Then we can apply (3) of Lemma \ref{trick-lemma-2} to $\pm(e_1-e_2)$, $\pm(e_2-e_3)$ and $\pm(e_1-e_3)$ to get a contradiction. Thus the corresponding normal space cannot be positively curved in this subcase.

\medskip
 In the following subcases we consider the situation that the angle between $\alpha$ and $\beta$ is $\frac{\pi}2$. Using  suitable Weyl group actions, one can  reduce the discussion to the following cases:
\begin{description}
\item{\rm (1)} $\alpha$ and $\beta$ are long roots, and we can assume that $\alpha=2e_1$
and $\beta=2e_2$;
\item{\rm (2)} $\alpha$ and $\beta$ are not of the same length, and we can assume
that $\alpha=e_1+e_2$ and $\beta=-2e_3$;
and when $\alpha$ and $\beta$ are short roots, we assume that either
\item{\rm (3)} $\alpha=e_1+e_2$ and $\beta=e_2-e_1$;
or
\item{\rm (4)}  $\alpha=e_1+e_2$ and $\beta=-e_3-e_4$.
\end{description}
Notice that (1) has been covered by the discussions in the previous subcase. Now we  discuss
the other cases in the following three subcases.

\medskip
\textbf{Subcase 2.}\quad Assume that $\alpha=e_1+e_2$ and $\beta=-2e_3$. Then
$\alpha'=\frac23(e_1+e_2-e_3)$ is a root of $\mathfrak{h}$
with  length
$\frac{2}{\sqrt{3}}$. By (2) of Lemma \ref{trick-lemma-2}, $\pm (e_1-e_2)$ are
roots of $\mathfrak{h}$.
Now none of the  vectors $\mathrm{pr}(\pm 2 e_i)$, $i=1, 2$,
is  orthogonal to
$\alpha'$, and each of them has length $\frac{\sqrt{30}}{3}$. Thus none of these vectors is a root of
$\mathfrak{h}$, i.e., $\mathfrak{g}_{\pm 2 e_i}\subset \mathfrak{m}$ for $i=1,2$.
Take any nonzero $u\in\mathfrak{g}_{\pm 2 e_1}$, $v\in\mathfrak{g}_{\pm 2 e_2}$, and two linearly
independent $w_1$ and $w_2$ from $\mathfrak{g}_{\pm 2 e_3}$. Then there are two Cartan
subalgebras, $\mathfrak{t}_1$ spanned by $u$, $v$, $w_1$ and all the elements $e_i$ with $i>3$;
and $\mathfrak{t}_2$ with the same linear generators except that  $w_1$ is changed by $w_2$. Then $\mathfrak{s}=\mathfrak{t}_1\cap\mathfrak{t}_2$ is a flat splitting subalgebra. This
is a contradiction to Theorem \ref{thm-4-2}.
%

\medskip
\textbf{Subcase 3.}\quad We assume that $\alpha=e_1+e_2$ and $\beta=e_2-e_1$. Then
$\alpha'=e_2$ is a root of $\mathfrak{h}$. By (2) of Lemma \ref{trick-lemma-2}, the root
$2e_2=2\alpha'\in\mathfrak{t}\cap\mathfrak{h}$ of $\mathfrak{g}$ is also a root of
$\mathfrak{h}$. This is a contradiction.

\medskip
\textbf{Subcase 4.}\quad We assume that $\alpha=e_1+e_2$ and $\beta=-e_3-e_4$.
Then the unit vector
$\alpha'=\frac12 e_1+\frac12 e_2-\frac12 e_3-\frac12 e_4$ is a root of $\mathfrak{h}$.
By (2) of Lemma \ref{trick-lemma-2}, $e_2-e_3$ is root of $\mathfrak{h}$ with  length $\sqrt{2}$. This implies that $\alpha'$ cannot be the root of   a simple factor of $\mathfrak{h}$
isomorphic to $G_2$.
Now none of the vectors $\mathrm{pr}(\pm 2 e_i)$ with
$1\leq i\leq 4$ is orthogonal to $\alpha'$, and they all have the same length
$\sqrt{3}$. So they are not roots of $\mathfrak{h}$, and $\mathfrak{g}_{\pm e_i}\subset
\mathfrak{m}$, for
$1\leq i\leq 4$.
Let  $v_i$ be a nonzero vector in $\mathfrak{g}_{\pm e_i}$. Then the elements  $v_i$
with $1\leq i\leq 4$,  and the elements $e_i$ with $4<i\leq n$, span a FSCS. This is a contradiction
to Theorem \ref{thm-4-2}.
Therefore in this subcase the normal homogeneous space cannot be positively curved.

\medskip
\textbf{Subcase 5.}\quad Now in the last subcase we consider the situation that the angle between $\alpha$ and $\beta$ is $\frac{3\pi}{4}$. Obviously we can assume that $\alpha=e_1+e_2$ and
$\beta=-2e_1$. Then $\alpha'=-\frac15 e_1+\frac35 e_2$ is a root of $\mathfrak{h}$ with
length $\frac{\sqrt{10}}{5}$. Now none of the
vectors $\mathrm{pr}(\pm e_i\pm e_j)$, with  $1\leq i\leq 2<j\leq n$,  is  orthogonal to $\alpha'$,
and each has a length $\frac{\sqrt{11}}{\sqrt{10}}$ or $\frac{\sqrt{19}}{\sqrt{10}}$. Thus none of these vectors   is
a root of $\mathfrak{h}$. Similarly, $\mathrm{pr}(\pm 2e_2)$ are not
roots of $\mathfrak{h}$. On the other hand, any  root of the subalgebra $\mathfrak{k}$ is either a linear
combination of $e_1$ and $e_2$, or a linear combination of the elements $e_i$ with $i>2$. Thus
$\mathfrak{k}$
is contained in $C_2\oplus C_{n-2}$ for $n>3$, or contained in $C_2\oplus A_1$ for $n=3$. If $n>3$, then Proposition \ref{summary-prop} indicates that
$G/K$ can not admit positively curved normal homogeneous Finsler metrics, which is
a contradiction.

If $n=3$, then we have $\mathfrak{g}=C_3$ and $\mathfrak{h}=A_1\oplus A_1$. Consider a linearly independent
commuting pair $u\in\mathfrak{g}_{\pm 2e_2}$ and
$v\in\mathfrak{g}_{\pm (e_1-e_3)}$ in $\mathfrak{m}$. From the
centralizer of $\mathfrak{s}=\mathbb{R}u+\mathbb{R}v\subset\mathfrak{m}$ in $\mathfrak{g}$, we can find nonzero vector $w_1=e_1+e_3$, and another nonzero vector $w_2$ from
$\mathfrak{g}_{\pm 2e_1}+\mathfrak{g}_{\pm 2e_3}$. Then the vectors $u$, $v$ and $w_1$
span a Cartan subalbegra $\mathfrak{t}_1$, and with $w_1$ changed to $w_2$, we have another $\mathfrak{t}_2$. So $\mathfrak{s}=\mathfrak{t}_1\cap\mathfrak{t}_2$ is a flat
splitting subalgebra of $\mathfrak{g}$. This is a contradiction to Theorem \ref{thm-4-2}.

To summarize, there is no positively curved normal homogeneous Finsler space in this subcase.

 Now we complete the proof of Theorem \ref{main1} and Theorem \ref{main2}. Up to equivalence, we have found all smooth coset spaces which may admit positively curved normal homogeneous Finsler metrics, as listed in Theorem \ref{main1}. They
are exactly those in Berger's classification work \cite{Ber61} (plus \cite{WILKING1990}), which does admit positive
normal homogeneous Riemannian metrics (which are just a special class of normal homogeneous Finsler metrics). This finishes the proof of Theorems \ref{main1} and \ref{main2}.
\section{Proof of Theorem \ref{main3}}
In this section we give a proof of Theorem \ref{main3}. It will be completed through a case by case study on the coset spaces appearing in Theorem \ref{main1}.

The coset spaces in (1) of Theorem \ref{main1} are Riemannian symmetric spaces of rank one.
For these spaces,  the isotropy group acts transitively
on the unit sphere $\mathcal{S}$ of $\mathfrak{m}$ with respect to any $\mathrm{Ad(G)}$-invariant inner product of $\mathfrak{g}$ (which is unqiue up to  homotheties). So
the corresponding normal homogeneous Finsler metrics  must be Riemannian.

On the other hand, in (2) of Theorem \ref{main1}, the above assertion for the isotropic action is also valid for the coset spaces $\mathrm{G}_2/\mathrm{SU}(3)$ and $S^3=\mathrm{SU}(2)/\mathrm{SU}(1)=\mathrm{Sp}(1)$, $S^6=G_2/\mathrm{SU}(3)$, and
$S^7=\mathrm{Spin}(7)/\mathrm{G}_2$. Therefore the normal homogeneous Finsler
metrics on the above coset spaces must also be Riemannian. This proves the first statement of Theorem \ref{main3}.

Now we turn to the second statement.  In the following we use the term ``the other spaces" to indicate the spaces in the lists of Theorem \ref{main1} which does not fall into the spaces considered in the above. To complete the proof,   we introduce the  Condition (R) which is vital
for our consideration. Let $G/H$ be a coset spaces of a compact Lie group $G$ and
$\mathfrak{g}=\mathfrak{h}+\mathfrak{m}$ be the reductive decomposition with respect to a normal
Riemannian metric. Denote the unit sphere of $\mathfrak{m}$ with respect to the induced inner product by $\mathcal{S}$.

\medskip
\noindent{\textbf{Condition (R):}
the unit sphere $\mathcal{S}$ in $\mathfrak{m}$ is contained in a single $\mathrm{Ad}(G)$-orbit in $\mathfrak{g}$.

\medskip
We now give a useful lemma. Recall that a Riemannian homogeneous manifold $G/H$ is
called isotropic, if the action of $H$ on the unit sphere $\mathcal{S}$ is transitive. Note that
an isotropic normal homogeneous Riemannian manifold must satisfy the condition (R), but it is unknown whether the converse is true.
\begin{lemma}
Let $G/H$ be a coset space of a compact Lie group $G$ with a normal homogeneous Riemannian
metric $Q$. If $G/H$ satisfies the condition (R), then any normal homogeneous Finsler metric
on $G/H$ must be a positive multiple of $Q$, hence must be Riemannian. On the other hand,
if $G/H$ does not satisfy the condition (R) and the Riemannian
metric $Q$ has positive sectional curvature, then there exist
non-Riemannian normal homogeneous Finsler metrics on $G/H$ with positive flag curvature.
\end{lemma}
\begin{proof}
The first assertion is obvious, since if $G/H$ satisfies the condition (R), then the restriction of the induced Minkowski norm of any normal homogeneous Finsler metric in $\mathcal{S}$ must be a constant multiple of the Euclidean norm of $Q$.

Now we prove the second assertion.
By the assumption, there exist two points $x_1,x_2\in \mathcal{S}$,
such that $\mathrm{Ad}(G)\cdot x_1\cap\mathrm{Ad}(G)\cdot x_2=\emptyset$.
By the parity of unit of smooth manifolds, there exists a
smooth nonnegative function $f_1$ on the unit sphere $\mathcal{S}'$ of $\mathfrak{g}$ with respect to the
inner product induced by $Q$, such that $f_1$ is equal to $1$ on
$\mathrm{Ad}(G)\cdot x_2$ and the support of $f_1$ is contained in
an open neighborhood $\mathcal{U}$ of $\mathrm{Ad}(G)\cdot x_2$ with
$\overline{\mathcal{U}}\cap\mathrm{Ad}(G)\cdot x_1=\emptyset$.
Now using the bi-invariant Haar measure of the Lie group $G$, we can define a
smooth nonnegative function $f$ on $\mathcal{S}'$ which is invariant
under the adjoint action of $G$ such that $f(\mathrm{Ad}(G)\cdot x_2)=1$ and there is an open subset of $\mathcal{U}_1$ containing $\mathrm{Ad}(G)\cdot x_2 $, which is invariant under the adjoint action of $G$, such that $\mathcal{U}_1\cap \mathrm{Ad}(G)\cdot x_1=\emptyset$ and $f(x)=0$, for any $x\notin \mathcal{U}_1$. Now for any positive number $\varepsilon $, we define
a real function on $\mathfrak{g}\backslash \{0\}$ by
$$\bar{F}_\varepsilon(x)=\sqrt{ \langle x, x\rangle+\varepsilon |x|^2f(\frac{x}{|x|}) }.$$
Then it is easy to check that for sufficiently small $\varepsilon$, $\bar{F}_\varepsilon$ is a positive definite
Minkowski norm on $\mathfrak{g}$ (see \cite{DE11}). Since $f$ is invariant under the adjoint action of $G$, $\bar{F}_\varepsilon$ is bi-invariant. On the other hand, when $\varepsilon$ is sufficiently close to $0$, the
indicatrix of the Minkowski norm $\bar{F}_\varepsilon$
coincides with $\mathcal{S}'$ around $x_1$, but differs with $\mathcal{S}'$
around $x_2$. Thus the corresponding normal homogeneous metric $F_\varepsilon$ on $G/H$
are non-Riemannian when $\varepsilon$ is positive and sufficiently small. By the assumption, $F_0$ has positive sectional curvature, so for
$\epsilon$ sufficiently close to $0$, $F_\epsilon$ has positive flag curvature.
\end{proof}

%

We now return to the proof of the there. We first   assert that for any of the other spaces such that the group $G$ in
$M=G/H$ is not simple, Condition (R) is not satisfied. In fact, we can re-scale the bi-invariant inner product differently on each simple factor or
Euclidean factor to get a family of induced inner product on $\mathfrak{m}$ which are not the same
up to scalar multiplications. Therefore for any nonzero vector $x\in\mathfrak{m}$, there exists a bi-invariant inner product $\langle,\,,\rangle_1$ on $\mathfrak{g}$, such
that the unit sphere $\mathcal{S}_1$ satisfies the conditions that $\mathcal{S}_1\cap \mathrm{Ad}(G)\cdot x\ne \emptyset$ and $\mathcal{S}_1\not\subset \mathrm{Ad}(G)\cdot x$. This proves the assertion.

Now we consider the other spaces in  Theorem \ref{main1} with simple Lie group $G$. We need some calculation.
For any Lie algebra $\mathfrak{g}$ of the Lie groups, there is a canonical way to present the elements in the Lie algebra $\mathfrak{g}$ by real
or complex
matrices. Let  $v\in\mathfrak{g}$. Then we define the {\it eigenvalue sequence} of $v$ to be the sequence of all
its eigenvalues of  the corresponding matrix, canonically ordered (notice that all the eigenvalues
belong to the imaginary line), and viewed as a vector. It is easy to see that, if Condition (R)
is satisfied,
then the eigenvalue sequences for each pair of vectors in
$\mathfrak{m}$ must be linearly dependent to each other. Therefore, if we can find a family of elements in $\mathfrak{m}$ such that some of the eigenvalue sequences of them are linearly independent, then the corresponding coset space does not satisfy Condition (R). In the following, we will
construct such  a family $v(t)$ of vectors in $\mathfrak{m}$, $t\in\mathbb{R}$, presented canonically by matrices.

First we consider  $M=\mathrm{SU}(n)/\mathrm{SU}(n-1)$ with $n>2$. With $\mathfrak{g}$ identified with
the matrix algebra $\mathfrak{su}(n)$, and $\mathfrak{h}$ identified with $\mathfrak{su}(n-1)\subset \mathfrak{su}(n)$,
corresponding to the left up block, we have a family of vectors
$$ v(t)=\left(
         \begin{array}{ccccc}
           -\sqrt{-1} & 0 & \cdots & 0 & 0 \\
           0 & \ddots & \ddots & \vdots & \vdots \\
           \vdots & \ddots & -\sqrt{-1} & 0 & 0 \\
           0 & \cdots & 0 & -\sqrt{-1} & t \\
           0 & \cdots & 0 & -t & (n-1)\sqrt{-1} \\
         \end{array}
       \right)\in\mathfrak{m}, \quad t\in\mathbb{R}.
$$

Next consider $M=\mathrm{Sp}(n)/\mathrm{Sp}(n-1)$ with $n>1$. With $\mathfrak{g}$ identified
with $\mathfrak{sp}(n)\subset \mathfrak{su}(2n)$, such that any quaternion number $a+b\mathbf{i}+c\mathbf{j}+
d\mathbf{k}$ is identified with the matrix $$\left(
                                         \begin{array}{cc}
                                           a+b\sqrt{-1} & c+d\sqrt{-1} \\
                                           -c+d\sqrt{-1} & a-b\sqrt{-1} \\
                                         \end{array}
                                       \right),$$
and $\mathfrak{h}$ identified with subalgebra for the left up block, we have
a family of vectors
$$ v(t)=\left(
          \begin{array}{ccccccc}
          0      & 0      & 0      & 0      &\cdots & 0         & 0 \\
          0      & \ddots & \ddots & \ddots &\ddots & \vdots    & \vdots \\
          0      & \ddots & 0      & 0      & 0     & 0         & 0 \\
          0      & \ddots & 0      & 0      & 0     & 0         & t \\
          \vdots & \ddots & 0      & 0      & 0     & -t        & 0 \\
          0      & \cdots & 0      & 0      & t     & \sqrt{-1} & 0 \\
          0      & \cdots & 0      & -t     & 0     & 0         & -\sqrt{-1} \\
          \end{array}
        \right)\in\mathfrak{m}, \quad t\in\mathbb{R}.
$$

Now we consider $M=\mathrm{Spin}(9)/\mathrm{Spin}(7)$. Identifying  $\mathfrak{g}$  with
the real matrix algebra $\mathfrak{so}(9)$, and applying our discussion about the root system of $\mathfrak{h}$
for this subcase, we have a family of vectors
$$v(t)=\left(
         \begin{array}{ccccccccc}
           0       & t      & 0    & 0 & 0 & 0 & 0 & 0 & 0 \\
           -t      & 0      & 1    & 0 & 0 & 0 & 0 & 0 & 0 \\
           0       & -1     & 0    & 0 & 0 & 0 & 0 & 0 & 0 \\
           0       & 0      & 0    & 0 & 1 & 0 & 0 & 0 & 0 \\
           0       & 0      & 0    & -1& 0 & 0 & 0 & 0 & 0 \\
           0       & 0      & 0    & 0 & 0 & 0 & 1 & 0 & 0 \\
           0       & 0      & 0    & 0 & 0 & -1& 0 & 0 & 0 \\
           0       & 0      & 0    & 0 & 0 & 0 & 0 & 0 & 1 \\
           0       & 0      & 0    & 0 & 0 & 0 & 0 & -1& 0 \\
         \end{array}
       \right)\in\mathfrak{m}, \quad t\in\mathbb{R},
$$
where the right down eight by eight block is from $\mathbb{R}(e_1+e_2+e_3+e_4)\in
\mathfrak{m}$, and the rest terms is from $\mathfrak{g}_{\pm e_1}\in\mathfrak{m}$.

Let us   consider the space $M=\mathrm{Sp}(2)/\mathrm{SU}(2)$. Using the matrix presentation given in
\cite{WILKING1990},  we have a family of vectors
$$v(t)=\left(
         \begin{array}{cc}
           \mathbf{i} & t\mathbf{j} \\
           t\mathbf{j} & -3\mathbf{i} \\
         \end{array}
       \right)
=\left(
   \begin{array}{cccc}
     \sqrt{-1} & 0 & 0 & t \\
     0 & -\sqrt{-1} & -t & 0 \\
     0 & t & -3\sqrt{-1} & 0 \\
     -t & 0 & 0 & 3\sqrt{-1} \\
   \end{array}
 \right)\in\mathfrak{m}, \quad t\in\mathbb{R}.
$$

Finally, consider the space $M=\mathrm{SU}(5)/\mathrm{Sp}(2)S^1$. Using the matrix presentation given
in \cite{WILKING1990}, and identifying
$$\left(
    \begin{array}{cc}
      a_1+b_1\mathbf{i}+c_1\mathbf{j}+d_1\mathbf{k} & a_2+b_2\mathbf{i}+c_2\mathbf{j}+d_2\mathbf{k} \\
      a_3+b_3\mathbf{i}+c_3\mathbf{j}+d_3\mathbf{k} & a_4+b_4\mathbf{i}+c_4\mathbf{j}+d_4\mathbf{k} \\
    \end{array}
  \right)\in \mathfrak{sp}(2)
$$
with
$$\left(
    \begin{array}{cccc}
      a_1+b_1\sqrt{-1}  & c_1+d_1\sqrt{-1} & a_2+b_2\sqrt{-1} & c_2+d_2\sqrt{-1} \\
      -c_1+d_1\sqrt{-1} & a_1-b_1\sqrt{-1} & -c_2+d_2\sqrt{-1} & a_2-b_2\sqrt{-1} \\
      a_3+b_3\sqrt{-1} & c_3+d_3\sqrt{-1}  & a_4+b_4\sqrt{-1} & c_4+d_4\sqrt{-1} \\
      -c_3+d+3\sqrt{-1} & a_3-b_3\sqrt{-1} & -c_4+d_4\sqrt{-1} & a_4-b_4\sqrt{-1} \\
    \end{array}
  \right)
$$
at the left up corner of $\mathfrak{su}(5)$,  we have a family of vectors
$$v(t)=\left(
         \begin{array}{ccccc}
           \sqrt{-1} & 0         & 0 & 0 & 0 \\
           0         & \sqrt{-1} & 0 & 0 & 0 \\
           0 & 0     & -\sqrt{-1}& 0 & 0 \\
           0 & 0 & 0 & -\sqrt{-1} & t \\
           0 & 0 & 0 & -t & 0 \\
         \end{array}
       \right)\in\mathfrak{m}, \quad t\in\mathbb{R}.
$$

Now a  direct calculation shows that, in each of the above cases, the eigenvalue sequence of the element $v(0)$ is
linearly independent to the eigenvalue sequence of $v(t)$ for $t\ne 0$ and $|t|$ sufficiently small.
Therefore the spaces in these cases do not satisfy  Condition (R).

The proof of Theorem \ref{main3} is now completed.

\begin{remark}
For those non-Riemannian normal homogeneous spaces listed in Theorem \ref{main1}, the properties of their isotropic
representations (i.e., the $\mathrm{ad}\mathfrak{h}$ actions on $\mathfrak{m}$) can
sometimes give us some more information on the types of invariant Finsler metrics on those spaces. For example, on the homogeneous
spheres $\mathrm{SU}(n)/\mathrm{SU}(n-1)$,
$\mathrm{U}(n)/\mathrm{U}(n-1)$, $\mathrm{Sp}(n)/\mathrm{Sp}(n-1)$ and $\mathrm{Sp}(n)S^1/\mathrm{Sp}(n-1)S^1$, any invariant Finsler metric must be an $(\alpha,\beta)$-metric.
On $\mathrm{Sp}(n)\mathrm{Sp}(1)/\mathrm{Sp}(n-1)\mathrm{Sp}(1)$, any invariant Finsler metric must be an $(\alpha_1,\alpha_2)$-metric (see \cite{XuDeng2014}).
\end{remark}
%
%

\end{document}